\documentclass[12pt]{article}
\usepackage{amsmath,amsthm,verbatim,amssymb,amsfonts,amscd,diagrams, graphics, mathrsfs}
\theoremstyle{plain}
\newtheorem{theorem}{Theorem}[section]
\newtheorem{corollary}[theorem]{Corollary}
\newtheorem{lemma}[theorem]{Lemma}
\newtheorem{proposition}[theorem]{Proposition}

\theoremstyle{definition}
\newtheorem{definition}[theorem]{Definition}
\newtheorem{question}[theorem]{Question}

\theoremstyle{remark}
\newtheorem{remark}[theorem]{Remark}

\newarrow{ul}---->
\newarrow{Backwards}<----

\newcommand{\td}[1]{\tilde{#1}}
\newcommand{\into}{\hookrightarrow}
\newcommand{\Z}{\mathbb{Z}}
\newcommand{\Q}{\mathbb{Q}}
\newcommand{\R}{\mathbb{R}}

\newcommand{\bd}{\partial}

\renewcommand{\H}{\mathbb H}

\newcommand{\mc}[1]{\mathcal{#1}}

\newcommand{\dlim}{\varinjlim}

\newcommand{\codim}{\text{codim}}

\newcommand{\mf}{\mathfrak}

\renewcommand{\i}{\mathfrak i}

\newarrow{Onto}----{>>}
\newarrow{Equals}=====
\newarrow{Into}C--->

\begin{document}

\title{Intersection homology K\"unneth theorems}
\author{Greg Friedman\\Texas Christian University}
\date{July 1, 2008}

\maketitle

\textbf{2000 Mathematics Subject Classification:} Primary:  55N33,  55U25;
Secondary:   57N80

\textbf{Keywords:} intersection homology, K\"unneth theorem, pseudomanifold

\begin{abstract}
Cohen, Goresky and Ji showed that there is a
K\"unneth theorem relating the intersection homology groups $I^{\bar p}H_*(X\times Y)$ to $I^{\bar p}H_*(X)$ and $I^{\bar p}H_*(Y)$, provided
that the perversity $\bar p$ satisfies rather strict conditions.  We consider
biperversities and prove that there is a K\"unneth theorem relating
$I^{\bar p,\bar q}H_*(X\times Y)$ to $I^{\bar p}H_*(X)$ and $I^{\bar q}H_*(Y)$ for all choices of $\bar p$ and $\bar q$. Furthermore, we prove that the K\"unneth theorem
still holds when the biperversity $p,q$ is ``loosened'' a little,
and using this we recover the K\"unneth theorem of Cohen-Goresky-Ji.
\end{abstract}

\section{Introduction}

Our goal in this paper is to study K\"unneth theorems for intersection homology. Intersection homology was developed by Goresky and MacPherson \cite{GM1} in the late 1970s for the purpose of studying stratified spaces. These are spaces more general than manifolds -- points might not possess euclidean neighborhoods -- but they are not ``too wild.'' In particular, stratified spaces are composed of layers of manifolds of various dimension, patched together in a well-controlled way. A precise definition of \emph{stratified pseudomanifolds} is given in the next section, but suffice to say that such spaces are abundant in nature, including large classes of algebraic and analytic varieties and quotients of manifolds by smooth group actions. 

Intersection homology has turned out to be a remarkably successful tool in studying such spaces. The original key result of Goresky and MacPherson was that intersection homology possesses a form of Poincar\'e duality over field coefficients, and thus one obtains for stratified spaces signature invariants, L-classes, etc. 
More wonderful results followed, including versions for singular varieties of the K\"ahler package  (hard Lefschetz, Lefschetz hyperplane, Hodge duality), substituting intersection homology for the ordinary homology appearing in the K\"ahler package for nonsingular varieties. Furthermore, it was not long before intersection homology was branching into other areas of mathematics, for example playing a key role in the proof of the Kazhdan-Lusztig conjecture in representation theory. Good surveys of these developments include \cite{KirWoo, Klei, Bo94}.

Despite these vast successes, some of the topological underpinnings of intersection homology remain somewhat mysterious. In its initial conception, the intersection homology groups were defined via a chain subcomplex of the the usual geometric chain complex on a space\footnote{I will be deliberately vague about categories of spaces and types of simplices (PL vs. singular) in this introduction.}. The idea is that one first assigns a set of \emph{perversity parameters} $\bar p$ to a stratified space;  our space is made up of layers of manifolds of various dimensions, and $\bar p$ assigns a number to each layer. The intersection chain complex $I^{\bar p}C_*(X)$ is then defined from the chain complex $C_*(X)$ by allowing only the chains whose intersection with each layer $X^i$ is not too great, as measured by the dimension $i$ and the perversity $\bar p$. If one chooses the perversities $\bar p$ to satisfy an appropriate set of restrictions, then the resulting homology groups have the desired properties mentioned above. 

Unfortunately, however, there are some drawbacks to this construction. For one thing, intersection homology is not a homology theory in the traditional sense: it does not satisfy the Eilenberg-Steenrod axioms as it is not a homotopy invariant (though it is  a stratum-preserving homotopy invariant, under the correct notion of stratum-preserving homotopy \cite{GBF3}). In fact, it was a nontrivial early result \cite{GM2} that intersection homology is a homeomorphism invariant, meaning, in particular, that the choice of stratification is irrelevant. Also, there have been issues  about the proper categorical framework for intersection homology. Some of these foundational issues were further, depending on your point of view, either obfuscated or alleviated by an early paradigm shift that largely replaced this chain theoretic approach to intersection homology with a point of view located in the derived category of sheaf complexes. This shift made it possible to bring to bear on the field some important heavy machinery, and it is this approach that has been so hugely successful for many of the advances already mentioned.

On the other hand, the success of the sheaf approach has led to the comparative neglect of the more geometric chain formulation of intersection homology. Nonetheless, there are further results to be uncovered at the topological foundations of the subject, including results not clearly obtainable from a purely sheaf theoretic perspective (at least not from one that neglects chains altogether, such as the Deligne sheaf formulation). In \cite{GBF17}, we were able to use a combination of chain and sheaf methods to extend Poincar\'e duality to homotopically stratified spaces (which we shall not discuss in detail here) and in \cite{GBF18}, we initiated a study of the algebraic structures of the intersection pairing of intersection chains on piecewise linear pseudomanifolds. The latter paper, as well as the present one, are parts of an ongoing collaboration with James McClure and Scott O. Wilson to study these algebraic structures.

\medskip

The present work concerns the K\"unneth property for intersection homology. 
Intersection homology does not possess a K\"unneth theorem in complete generality in the sense that $I^{\bar p}H_*(X\times Y)\cong H_*(I^{\bar p}C_*(X)\otimes I^{\bar p}C_*(Y))$ for any perversity $\bar p$ and pseudomanifolds $X,Y$, though K\"unneth properties do hold in certain situations. Special cases have been proven by Cheeger \cite{Chee80}, Goresky-MacPherson \cite{GM1, GM2},  Siegel \cite{Si83}, King \cite{Ki}, and Cohen, Goresky, and Ji \cite{CGJ}, culminating in two essentially distinct results. The first, proved in the greatest generality by King \cite{Ki}, is the fact that 
$I^{\bar p}H_*(M\times X)\cong H_*(C_*(M)\otimes I^{\bar p}C_*(X))$ when $X$ is a pseudomanifold and $M$ is a  manifold.

The other most general result in this area appears in Cohen, Goresky, and Ji \cite{CGJ}. Along with providing counterexamples to the existence of a general K\"unneth theorem for a single perversity, they show that   $ I^{\bar p}H_*(X\times Y;R)\cong H_*(I^{\bar p}C_*(X;R)\otimes I^{\bar p}C_*(Y;R))$ for pseudomanifolds $X$ and $Y$ and a principal ideal domain $R$, provided either that $\bar p(a)+\bar p(b)\leq \bar p(a+b)\leq \bar p(a)+\bar p(b)+1$ for all $a$ and $b$ \emph{or} that $\bar p(a)+\bar p(b)\leq \bar p(a+b)\leq \bar p(a)+\bar p(b)+2$ for all $a$ and $b$  and either $X$ or $Y$ is \emph{locally $\bar p$-torsion free over $R$}. This last condition ensures the vanishing of torsion in certain local intersection homology groups.

Observe that all of these past results constrain themselves to a single perversity $\bar p$. 

We proceed somewhat in the opposite direction of the K\"unneth-type results stated above. We answer the following question posed by James McClure: Given PL pseudomanifolds $X$ and $Y$ and perversities $\bar p$ and $\bar q$, is there a chain complex defined geometrically on $X\times Y$ whose homology is isomorphic to that of $I^{\bar p}C_*(X)\otimes I^{\bar q}C_*(Y)$? 
This question is motivated by the desire to find a cup product for intersection cohomology with field coefficients suitably dual to the Goresky-MacPherson intersection pairing $I^{\bar p}H_*(X)\otimes I^{\bar q}H_*(X)\to I^{\bar r}H_*(X)$, where $\bar p(k)+\bar q(k)\leq \bar r(k)$ for all $k$. Equivalently, we would like to find a ``diagonal map'' of the form $IH_*(X)\to IH_*(X)\otimes IH_*(X)$ on intersection homology with appropriate perversities (at least with field coefficients) whose dual would constitute a cup product.  However, the Alexander-Whitney map is unavailable in this context because it
does not preserve the admissibility conditions for intersection chains. An alternative approach \emph{in ordinary homology} is to define the diagonal map (with field coefficients)
as the composite
\[
H_*(X)
\to
H_*(X\times X)
\stackrel{\cong}{\leftarrow}
H_*(X)\otimes H_*(X),
\]
where the first map is induced by the geometric diagonal inclusion map and the second is the Eilenberg-Zilber shuffle product,
which is an isomorphism by the K\"unneth theorem (note that the shuffle product
should have better geometric properties than the Alexander-Whitney map because
it is really just Cartesian product).  This suggests the problem of doing
something similar in intersection homology, and the first step is a suitable
K\"unneth theorem. 

In the following, we show that, in fact, there are several choices of ``biperversities'' associated to the product bifiltration of $X\times Y$ that yield the desired results. The precise statement can be found in Theorem \ref{T: Kunneth} on page \pageref{T: Kunneth}, below; see also Corollary \ref{C: JH}, which contains an important special case that may be easier to absorb on a first pass. The local calculations employed in our proof are similar to those employed by Cohen, Goresky, and Ji in \cite{CGJ}. However, we observe that while their  theorem  seems to indicate that the K\"unneth theorem
for intersection homology can only be expected to work in special cases, with a
rather strong hypothesis on the allowed perversities,  the natural generalizations we employ, extending filtrations to bifiltrations and single perversities to pairs of perversities, yields a theory  (in fact several theories) that hold in complete generality, requiring no extra conditions on the perversities $\bar p,\bar q$ or on the spaces involved. Both the Cohen-Goresky-Ji result and the K\"unneth theorem for which one term is a manifold appear as corollaries to our Theorem \ref{T: Kunneth} (though we do use in the proof the special case in which the manifold is $\R^n$).

To introduce the idea of our main theorem, Theorem \ref{T: Kunneth}, we present here a special case  that is simpler to state than the full the result.

\begin{corollary}\label{C: JH}
Let $X^m,Y^n$ be stratified pseudomanifolds and $\bar p,\bar q$ traditional perversities. If we let $I^{\bar p,\bar q}H^c_*(X\times Y)$ be a singular homology theory defined just as ordinary intersection homology but using the perversity $\bar p(k)+\bar q(l)$ on the product stratum $X_{m-k}\times Y_{n-l}$, then  
 $I^{\bar p,\bar q}H_i(X\times Y)$ is isomorphic to $$H_i(I^{\bar p}C^c_*(X)\otimes I^{\bar q}C^c_*(Y)) \cong \bigoplus_{j+k=i}I^{\bar p}H^c_j(X)\otimes I^{\bar q}H^c_k(Y)\oplus \bigoplus_{j+k=i-1}I^{\bar p}H^c_j(X;R)* I^{\bar q}H^c_k(Y) .$$ 
\end{corollary}

It is interesting to observe that, while we make use of some sheaf machinery, it remains unclear how one would formulate our product intersection homology theory in purely sheaf theoretic terms along the lines of the Deligne construction (how would we truncate here?). Such a formulation would be very interesting to have. 

In Section \ref{S: background}, we provide official definitions and background material. In Section \ref{S: Kunneth}, we formulate and present our main theorem, Theorem \ref{T: Kunneth}. Section \ref{S: proof} contains the proof. Section \ref{S: super} concerns how Theorem \ref{T: Kunneth} must be modified if one wishes to take into account more general perversities than the traditional ones. Finally, we provide an Appendix, which contains some technical details regarding intersection homology with these nontraditional perversities. 

We will work throughout in the setting of singular intersection homology on topological pseudomanifolds (which is equivalent to the sheaf theoretic intersection cohomology of Goresky-MacPherson - see \cite{GBF10}), but the same results can be proven in the PL category using PL chains.

\paragraph{Acknowledgment.} I thank Jim McClure for his invaluable insight in suggesting this project and all his help along the way. I also thank Scott Wilson for much helpful correspondence.

\section{Background}\label{S: background}

In this section, we recall some background material. The reader anxious to get to the main results can skip ahead, referring here for details.  The reader interested in further background on intersection homology might consult the original papers by Goresky and MacPherson \cite{GM1,GM2} and the very thorough notes of Borel, et. al. \cite{Bo}. For an alternative introduction and an overview of the applications of intersection homology to other fields of mathematics, the reader should see Kirwan and Woolf \cite{KirWoo} or Banagl \cite{BaIH}. More details on singular intersection homology can be found in \cite{Ki, GBF10}.

\paragraph{Pseudomanifolds.} Let $c(Z)$ denote the open cone on the space $Z$,
and let $c(\emptyset)$ be a point. 

A \emph{stratified paracompact Hausdorff space}
$Y$ (see \cite{GM2} or \cite{CS}) is defined
by a filtration
\begin{equation*}
Y=Y^n\supset Y^{n-1} \supset Y^{n-2}\supset \cdots \supset Y^0\supset Y^{-1}=\emptyset
\end{equation*}
such that for each point $y\in Y_i=Y^i-Y^{i-1}$, there exists a \emph{distinguished neighborhood}
$U$ of $y$ such that there is a \emph{compact} Hausdorff space $L$, a filtration  of $L$
\begin{equation*}
L=L^{n-i-1}\supset  \cdots \supset L^0\supset L^{-1}=\emptyset,
\end{equation*}
and a homeomorphism
\begin{equation*}
\phi: \R^i\times c(L)\to U
\end{equation*}
that takes $\R^i\times c(L^{j-1})$ onto $Y^{i+j}\cap U$. The subspace $Y_i=Y^i-Y^{i-1}$ is called the $i$th \emph{stratum}, and, in particular, it is a (possibly empty) $i$-manifold. $L$ is called a \emph{link}.

A stratified \emph{(topological) pseudomanifold} of dimension $n$ is a stratified paracompact Hausdorff space $X$  such that $X^{n-1}=X^{n-2}$,  $X-X^{n-2}$ is a manifold of dimension $n$ dense in $X$, and each link $L$ is, inductively, a stratified pseudomanifold. A space is a (topological) pseudomanifold if it can be given the structure of a stratified pseudomanifold for some choice of filtration. Intersection homology is known to be a topological invariant of such spaces; in particular, it is invariant under choice of  stratification (see \cite{GM2}, \cite{Bo}, \cite{Ki}). Examples of pseudomanifolds include complex algebraic and analytic varieties.

We refer to the link $L$ in the neighborhood $U$ of $y$ as the link of $y$ or of the component of the stratum containing $y$; it is, in general, not uniquely determined up to homeomorphism, though if $X$ is a pseudomanifold it is unique up to, for example, stratum preserving homotopy equivalence (see, e.g., \cite{GBF13}), which is sufficient for the intersection homology type of the link of a stratum component to be determined uniquely. Thus there is no harm, in general, of referring to ``the link'' of a stratum component instead of ``a link'' of  a stratum component. 

The following (well-known) lemma will be useful.

\begin{lemma}\label{L: link link}
Let $X$ be a pseudomanifold, and let $L=L^{k-1}$ be the link of $x\in X_{n-k}$. Let $\mf L$ be the link of a point $y\in L_{k-1-u}$. Then $\mf L$ is a link of the codimension $u$ stratum of $X$. 
\end{lemma}
\begin{proof}
By assumption, $x\in X_{n-k}$ has a distinguished neighborhood of the form $N\cong\R^{n-k}\times cL$, and, furthermore, $N\cap X_{n-u}\cong \R^{n-k}\times L_{k-1-u}\times (0,1)\subset \R^{n-k}\times cL_{k-1-u}$. So now let $y\in L_{k-1-u}$ with a neighborhood in $L$ of the form $\R^{k-1-u}\times c\mf L$. But  now we can identify $L$ with a particular factor of $L$ in the product $\R^{n-k}\times L \times (0,1)\subset \R^{n-k}\times cL\cong N$. Then $y\in X_{n-u}$ with a neighborhood of the form $\R^{n-k}\times \R^{k-1-u}\times c\mf L \times (0,1)\cong \R^{n-u}\times c\mf L$. The compatibility of the stratifications is clear, so $\mf L$ is a link of $y$ in $X$.  
\end{proof}

\paragraph{Intersection homology.}

In this section, we provide a quick review of the definition of intersection homology. For more details, the reader is urged to consult King \cite{Ki} and the author \cite{GBF10} for singular intersection homology and the original papers of Goresky and MacPherson \cite{GM1,GM2} and the book of Borel \cite{Bo} for the simplicial and sheaf definitions. Singular chain intersection homology theory was introduced in \cite{Ki} with finite chains (compact supports) and generalized in \cite{GBF10} to include locally-finite but infinite chains (closed supports). 

We recall that singular  intersection homology  can be defined on any filtered space 

$$X=X^n\supset X^{n-1}\supset \cdots \supset X^0\supset X^{-1}=\emptyset.$$
In general, the superscript ``dimensions'' are simply labels and do not necessarily reflect any geometric notions of dimension. We refer to $n$ as the \emph{filtered dimension} of $X$, or simply as the ``dimension'' when no confusion should arise. 
The set $X^i$ is called the $i$th \emph{skeleton} of $X$, and $X_i=X^i-X^{i-1}$ is the $i$th \emph{stratum}. Of course when $X$ is a pseudomanifold, the index $i$ will represent the  dimension of the stratum in the usual sense.

A \emph{perversity} $\bar p$ is a function $\bar p: \Z^{\geq 0}\to \Z$ such that $\bar p(k)\leq \bar p(k+1)\leq \bar p(k)+1$. A \emph{traditional perversity} also satisfies $\bar p(0)=\bar p (1)=\bar p(2)=0$; in particular, for $k\geq 2$, $\bar p(k)\leq k-2$. One generally must restrict to traditional perversities in order to obtain the most important topological invariance and Poincar\'e duality results for intersection homology (see \cite{GM2, Bo, Ki, Q2}), although many interesting results are now also known for superperversities, which satisfy $\bar p(2)>0$ (see \cite{CS, HS91, GBF10, GBF11, Sa05}). King \cite{Ki} also considers \emph{loose perversities}, which are completely arbitrary functions $\bar p: \Z^{\geq 0}\to \Z$, and one can generalize these even further to be functions $\bar p: \{\text{connected components of strata of $X$} \}\to \Z$. We will consider loose perversities here as well, though we do impose the one condition $\bar p(0)=0$ (or $\bar p(U)=0$ if $U$ is a component of the top stratum).
The reasons for this are that, on the one hand, if we allowed $\bar p(0)<0$, chains could not intersect the top dense stratum, which is a degenerate situation it is reasonable to avoid, while there is no loss of generality in reducing $\bar p(0)\geq 0$ to $\bar p(0)=0$ (see the definition of intersection homology, below). Making $\bar p(0)=0$ a blanket assumption now saves us from having to make some later statements unnecessarily complicated.

Given $\bar p$ and $X$ and a coefficient ring $R$, one defines the \emph{intersection chain complex} $I^{\bar p}C^c_*(X;R)$ as a subcomplex of $C^c_*(X;R)$, the complex of compactly supported singular chains\footnote{This is the usual chain complex consisting of finite linear combination of singular simplices, but we emphasize the compact supports in the notation to distinguish $C^c_*(X)$ from $C^{\infty}_*(X)$, which we shall also use.} on $X$, as follows: A singular $i$-simplex $\sigma:\Delta^i\to X$ is \emph{allowable} if $$\sigma^{-1}(X^{n-k}-X^{n-k-1})\subset \{i-k+\bar p(k) \text{ skeleton of } \Delta^i\}.$$ The chain $\xi\in C^c_i(X;R)$ is allowable if each simplex in $\xi$ and $\bd \xi$ is allowable. $I^{\bar p}C_*^c(X;R)$ is the complex of allowable chains. $I^{\bar p}C_*^{\infty}(X;R)$ is defined similarly as the complex of allowable chains in $C_*^{\infty}(X;R)$, the complex of locally-finite singular chains. Chains in $C_*^{\infty}(X;R)$ may be composed of an infinite number of simplices (with their coefficients), but for each such chain $\xi$, each point in $X$ must have a neighborhood that intersects only a finite number of simplices (with non-zero coefficients) in $\xi$.  $I^{\bar p}C_*^{\infty}(X;R)$ is referred to as the complex of intersection chains with \emph{closed supports}, or sometimes as \emph{Borel-Moore} intersection chains. See \cite{GBF10} for more details. 

The associated homology theories are denoted $I^{\bar p}H^c_*(X;R)$ and $I^{\bar p}H^{\infty}_*(X;R)$ and called \emph{intersection homology} with, respectively, compact or closed supports. We will sometimes omit the decorations $c$ or $\infty$ if these theories are equivalent, e.g. if $X$ is compact, or for statements that apply equally well in either context. We will occasionally omit explicit reference to $\bar p$ in statements that hold for any fixed perversity. We also often leave the coefficient $R$ tacit.

Relative intersection homology is defined similarly, though we note that 
\begin{enumerate}
\item the filtration on the subspace will always be that inherited from the larger space by restriction, and
\item  in the closed support case, all chains are required to be locally-finite in the larger space. 
\end{enumerate}

If $(X,A)$ is such a filtered space pair, we use the notation $IC_*^{\infty}(A_X)$ to denote the allowable singular chains supported in $A$ that are locally-finite in $X$. The homology of this complex is $IH_*^{\infty}(A_X)$. Note that in the compact support case, the local-finiteness condition is satisfied automatically so we do not need this notation and may unambiguously refer to $IH_*^c(A)$. The injection  $0\to IC_*^{\infty}(A_X)\to IC_*^{\infty}(X)$ yields a quotient complex $IC_*^{\infty}(X,A)$ and a long exact sequence of intersection homology groups $\to IH_i^{\infty}(A_X)\to IH_i^{\infty}(X)\to IH_i^{\infty}(X,A)\to$.

The crucial local property of intersection homology, which we will use below, are the following formulas; see \cite{Ki,GBF10} for proofs.

\begin{proposition}\label{P: cone}
Let $L$ be an $n-1$ dimensional filtered space with coefficient ring $R$, and let $\bar p$ be a (possibly loose) perversity such that $\bar p(n)\leq n-2$. Then \begin{equation*}
I^{\bar p}H^{c}_i(cL;R) \cong
\begin{cases}
I^{\bar p}H^{c}_{i}(L;R), & i<n-1-\bar p(n),\\
0, & i\geq n-1-\bar p(n).
\end{cases}
\end{equation*}
If $L$ is compact, then 
\begin{equation*}
I^{\bar p}H^c_i(cL,L\times \R;R)\cong I^{\bar p}H^{\infty}_i(cL;R) \cong
\begin{cases}
I^{\bar p}H_{i-1}(L;R), & i\geq n-\bar p(n),\\
0, & i<n-\bar p(n).
\end{cases}
\end{equation*}
\end{proposition}

\begin{remark}\label{R: strat coef}
Unfortunately, Proposition \ref{P: cone}, which is fundamental for many intersection homology computations, will no longer be true if  $\bar p(n)> n-2$ for some $n\geq 2$.
We will revisit this issue in Section \ref{S: super}, below.

\end{remark}

\section{Product intersection homology theories}\label{S: Kunneth}

We can now state our main results.

We fix loose perversities $\bar p,\bar q$ (recall that we do assume $\bar p(0)=\bar q(0)=0$). Initially, we will also require that $\bar p(k)\leq k-2$ for all $k>0$, and similarly for $\bar q$, though we will loosen this condition in Section \ref{S: super}. Our spaces will be stratified pseudomanifolds $X^m$ and $Y^n$. We also fix  a principal ideal domain $R$ and often omit it from notation. We let $\pi_1:X\times Y\to X$ and $\pi_2:X\times Y\to Y$ be the projections.

We seek to find chain complexes on $X\times Y$ whose homology groups are isomorphic to\footnote{All tensor products are over $R$, and we will use $A*B$ throughout to denote the torsion product $Tor^1_R$ of the $R$-modules $A$ and $B$.} 
$$H_i( I^{\bar p}C^c_*(X;R)\otimes I^{\bar p}C^c_*(Y;R))\cong \bigoplus_{j+k=i} I^{\bar p}H_j(X;R)\otimes I^{\bar q}H_k(Y;R)\oplus \bigoplus_{j+k=i-1} I^{\bar p}H_j(X;R)* I^{\bar q}H_k(Y;R).$$

We will look at subcomplexes of the singular chain complex $C_*(X\times Y;R)$ and apply loose perversity conditions to connected components of strata of $X\times Y$. In general, two stratum components of the same dimension need not have the same perversity conditions. 

More precisely, we define a \emph{loose product perversity} $Q$ to be a function $$Q:\Z^{\geq 0}\times \Z^{\geq 0}\to\Z.$$ We then define a singular simplex $\sigma: \Delta^i\to X\times Y$ to be $Q$-allowable if, for each stratum component $S\subset X_{m-k}\times Y_{n-l}$ of $X\times Y$, we have $\sigma^{-1}(S)\subset \{i-k-l+Q(k,l) \text{ skeleton of }\Delta^i\}$. A chain is $Q$-allowable if each of its simplices with non-zero coefficient and each of the simplices with non-zero coefficient in its boundary are $Q$-allowable. In this way, we obtain chain complexes $I^QC^c_*(X\times Y;R)$ and $I^QC^{\infty}_*(X\times Y;R)$ and associated homology theories $I^QH_*^c(X\times Y;R)$ and $I^QH_*^{\infty}(X\times Y;R)$.

In this setting, our question becomes 
\begin{question}
For what $Q$ and for what conditions on $X,Y$ is $I^QH^c_*(X\times Y;R)\cong H_i( I^{\bar p}C^c_*(X;R)\otimes I^{\bar q}C^c_*(Y;R))$ ?
\end{question}

The answer is contained in the following theorem, the proof of which is presented in the next section.

\begin{theorem}\label{T: Kunneth}
If $\bar p$ and $\bar q$ are perversities such that $\bar p(k)\leq k-2$ and $\bar q(l)\leq l-2$ for all $k,l>0$, then $I^QH_*^c(X\times Y;R)\cong H_*(I^{\bar p}C_*^c(X^m;R)\otimes I^{\bar q}C_*^c(Y^n;R))$ if the following conditions hold:
\begin{enumerate}
\item $Q(k,0)=\bar p(k)$ and $Q(0,l)=\bar q(l)$ for all $k,l$,
\item For each pair $k,l$ such that $0<k\leq m$, $0<l\leq n$, either 
\begin{enumerate}
\item \label{I: simple case} $Q(k,l)=\bar p(k)+\bar q(l)$ \emph{or},
\item  $Q(k,l)=\bar p(k)+\bar q(l)+1$ \emph{or},
\item \label{I} $Q(k,l)=\bar p(k)+\bar q(l)+2$ \emph{and} $I^{\bar p}H_{k-2-\bar p(k)}(L_1;R)*I^{\bar q}H_{l-2-\bar q(l)}(L_2;R)=0$, where $L_1, L_2$ are the links of the codimension $k,l$ strata of $X,Y$, respectively. 
\end{enumerate}
\end{enumerate}

Furthermore, if these conditions are not satisfied, then $I^QH^c_*(X\times Y;R)$ will not equal $H_*(I^{\bar p}C_*^c(X^m;R)\otimes I^{\bar q}C_*^c(Y^n;R))$ in general. 
\end{theorem}

\begin{remark}
Corollary \ref{C: JH}, stated in the introduction, is what we get by limiting ourselves to possibility \eqref{I: simple case}.
\end{remark}

\begin{remark}
The theorem remains true, with the obvious modifications, if $\bar p$ and $\bar q$ are the more general types of perversities that are allowed to take different values on different connected components of strata of the same dimension, so long as $\bar p(S)\leq \codim(S)-2$ when $\codim(S)>0$ and $\bar p(S)= 0$ when $\codim(S)=0$, and similarly for $\bar q$. The generalization of the proof below is straightforward, but we avoid the notational complications that would be necessary. Also, the results of Section \ref{S: super} apply to these sorts of perversities as well. 
\end{remark}

\begin{remark} 
Recall from \cite{GS83} that a pseudomanifold $X$ is called  \emph{locally $\bar p$-torsion free} with respect to $R$ if, for each link $L$ of a stratum component of codimension $k$, the  torsion subgroup of $I^{\bar p}H_{k-2-\bar p(k)}(L)$ vanishes. We observe that the torsion product condition in case \eqref{I} of Theorem \ref{T: Kunneth} will be satisfied  if $X$ is locally $\bar p$-torsion free over $R$ or if $Y$ is locally $\bar q$-torsion free over $R$. In particular, it will be satisfied for any $X$, $Y$ if $R$ is a field. Thus the requirement on $X$ and $Y$ in   \eqref{I} is a fairly reasonable condition to consider. 
\end{remark}

The K\"unneth theorem of Cohen, Goresky, and Ji \cite{CGJ} and the K\"unneth theorem with a manifold factor  arise immediately as special cases of Theorem \ref{T: Kunneth}:
\begin{corollary}[Cohen-Goresky-Ji \cite{CGJ}]
If $\bar p$ is a traditional perversity and $ \bar p(k)+\bar p(l)\leq\bar p(k+l)\leq \bar p(k)+\bar p(l)+1$, then $I^{\bar p}H^c_*(X\times Y)\cong H_*(I^{\bar p}C_*^c(X)\otimes I^{\bar p}C_*^c(Y))$. Furthermore, with field coefficients or if either $X$ or $Y$ is locally $\bar p$-torsion free, then this condition can be weakened to $ \bar p(k)+\bar p(l)\leq\bar p(k+l)\leq \bar p(k)+\bar p(l)+2$.
\end{corollary}
\begin{proof}
Take $Q(k,l)=\bar p(k+l)$ in Theorem \ref{T: Kunneth}. Note that if $X_k\times Y_l$ and $X_{k'}\times Y_{l'}$ are stratum components of the same overall codimension, i.e. $k+l=k'+l'$, then $Q(k,l)=Q(k',l')=\bar p(k,l)$, so that indeed, $I^QH_*(X\times Y)$ is just the ordinary intersection homology $I^{\bar p}H_*(X\times Y)$. 
\end{proof}

\begin{corollary}[King \cite{Ki}]
If $M$ is a manifold, then $I^{\bar p}H^c_*(X\times M)\cong H_*(I^{\bar p}C_*^c(X)\otimes C_*^c(M))$.
\end{corollary}
\begin{proof}
Take $Q(k,l)=\bar p(k)=\bar p(k)+\bar 0(l)$, where $\bar 0$ is the perversity such that $\bar 0(l)=0$ for all $l$. We have $I^QH_*(X\times M)\cong I^{\bar p}H_*(X\times M)$ because all strata of $X\times M$ are of the form $X_k\times M$, so only the values $Q(k,0)=\bar p(k)$ come into play. On the other hand, 
recall \cite{GM2} that for any perversity $\bar q$, in particular $\bar 0$, $I^{\bar q}C_*(M)$ is quasi-isomorphic to $C_*(M)$.
\end{proof}

\section{Proof of Theorem \ref{T: Kunneth}}\label{S: proof}

In this section, we prove Theorem \ref{T: Kunneth}. We begin with some preliminaries about sheaves of intersection chains. While  we fundamentally work with intersection homology from the chain complex point of view (as opposed to axiomatics in the derived category of sheaves, for which see \cite{GM2,Bo,BaIH}, etc.), it is useful nonetheless to utilize sheaves of chain complexes in our framing argument. Recall from \cite{GBF10} that one can define a sheaf complex $\mc I^{\bar p}\mc C^*$ on the $m$-dimensional filtered space $X$ as the sheafification of the presheaf $U\to I^{\bar p}C^{\infty}_{m-*}(X,X-\bar U)$ or, equivalently, the presheaf $U\to I^{\bar p}C^{c}_{m-*}(X,X-\bar U)$. It is shown in \cite{GBF10} that the former presheaf is conjunctive for coverings and has no non-trivial global sections with empty support. Furthermore, the sheaf $\mc I^{\bar p}\mc C^*$  is homotopically fine. As a consequence, the hypercohomology $\H^i(X;\mc I^{\bar p}\mc C^*)$ is isomorphic to $I^{\bar p}H^{\infty}_{m-i}(X)$, and $\H_c^i(X;\mc I^{\bar p}\mc C^*)\cong I^{\bar p}H^{c}_{m-i}(X)$. These statements also hold using loose perversities. 

\subsection{Sheaf preliminaries} 
Even in the present setting of very general perversities on product spaces, the basic results of loose perversity singular chain intersection homology established in \cite{Ki} and \cite{GBF10} still hold for $I^QH_*$, such as stratum-preserving homotopy equivalence, excision, the K\"unneth theorem for which one term is $\R^n$ endowed with perversity $0$ on its unique stratum, the cone formula stated above, Mayer-Vietoris sequences, etc. The proofs of these kinds of results do not rely on the form of the perversity or that it be the same for all stratum components of the same dimension. 

Similarly, the standard procedures for creating sheaves of intersection chains and realizing intersection homology as the hypercohomology of these sheaves continue to apply. In particular, we can form a sheaf complex $\mc{I}^Q\mc C^*$ as the sheafification either of the presheaf $I^QC^*:U\to I^QC_{m+n-*}^c(X\times Y, X\times Y-\bar U;R)$ or of the presheaf $U\to I^QC_{m+n-*}^\infty(X\times Y, X\times Y-\bar U;R)$. The latter presheaf is conjunctive for coverings and has no nontrivial global section with empty support, and the  sheaf $\mc{I}^Q\mc C^*$ is homotopically fine. So $\H^*(X\times Y;\mc{I}^Q\mc C^*)\cong H^*(\Gamma(X\times Y;\mc{I}^Q\mc C^*))\cong I^QH^{\infty}_{n+m-*}(X\times Y;R)$ and similarly with compact supports. The arguments are precisely the same as in \cite{GBF10} for the more standard intersection chain complexes. If we were working with PL chains, we could define instead the sheaf of PL intersection chains in the usual manner, and this sheaf would be soft, as shown for the standard intersection chains in \cite[Proposition II.5.1]{Bo} (see also the excision arguments in \cite{GBF10}).

The main step in the proof of  Theorem \ref{T: Kunneth} will be to show that, under the hypotheses of the theorem, $\mc{I}^Q\mc C^*$ is quasi-isomorphic to  $\pi_1^*(\mc I^{\bar p}\mc C_X^*)\otimes \pi_2^* (\mc I^{\bar q}\mc C_Y^*)$, where $\pi_1,\pi_2$ are the projections of $X\times Y$ to $X$ and $Y$ and $\mc I^{\bar p}\mc C_X^*$ and $\mc I^{\bar q}\mc C_Y^*$ are the sheaves of singular intersection chains over $X$ and $Y$, respectively. To see why this suffices, 
we need the following useful lemma and its consequences.

\begin{lemma}\label{L: flat}
$\mc I^{\bar p}\mc C_X^*$ is flat, and thus so is $\pi_1^*(\mc I^{\bar p}\mc C_X^*)$. Therefore, $\pi_1^*(\mc I^{\bar p}\mc C_X^*)\otimes \pi_2^* (\mc I^{\bar q}\mc C_Y^*)$ represents the left derived functor $\pi_1^*(\mc I^{\bar p}\mc C_X^*)\overset{L}{\otimes} \pi_2^* (\mc I^{\bar q}\mc C_Y^*)$.
\end{lemma}
\begin{proof}
The pullback of a flat sheaf is flat, since flatness is a property of the stalks (see, e.g. \cite[V.6.1]{Bo}). Thus we need only show that $\mc I^{\bar p}\mc C_X^*$ is flat. By \cite{GBF10}, $\mc I^{\bar p}\mc C_X^*$ can be defined as the sheafification of the presheaf $U\to I^{\bar p}C^c_*(X,X-\bar U)$. 

We note that the obvious homomorphism induced by inclusion $j:I^{\bar p}C^c_*(X,X-\bar U)\to C^c_*(X,X-\bar U)$ is injective, since if $\xi$ is an allowable chain representing an element of $I^{\bar p}C^c_*(X,X-\bar U)$ and $j(\xi)=0$, then $|\xi|\in X-\bar U$, and thus $|\xi|=0\in I^{\bar p}C^c_*(X,X-\bar U)$. So $I^{\bar p}C^c_*(X,X-\bar U)$ is isomorphic to a submodule of $C^c_*(X,X-\bar U)$. We claim that this latter module is free, and thus $I^{\bar p}C^c_*(X,X-\bar U)$ is free, since submodules of free modules over principal ideal domains are free (see, e.g \cite[Theorem III.7.1]{LANG}). The lemma will then follow, since a sheaf derived from a flat presheaf is flat, as the direct limit functor is exact and commutes with tensor products. 

 To verify the claim, we observe that there is a splitting $r:C^c(X)\to C^c(X-\bar U)$ defined by taking singular simplices that do not have support in $X-\bar U$ to $0$ and by acting as the identity on singular simplices that do have support in $X-\bar U$. This is clearly a right inverse to the inclusion $C^c_*(X-\bar U)\into C^c_*(X)$. It follows that the quotient $C^c_*(X,X-\bar U)$ is isomorphic to a direct summand of $C^c_*(X)$, which is free, and so $C^c_*(X,X-\bar U)$ is also free. 
 
 (Note that this splitting argument does not work directly on intersection chains since subchains of allowable chains are not always allowable).
\end{proof}

\begin{corollary}
$$ \H_c^{n+m-i}(X\times Y;\pi_1^*(\mc I^{\bar p}\mc C_X^*)\otimes \pi_2^* (\mc I^{\bar q}\mc C_Y^*) )\cong H_i(I^{\bar p}C^c_*(X)\otimes I^{\bar q}C^c_*(Y))$$ 
\end{corollary}
\begin{proof}
By definition,  $\H^i_c(X\times Y;\pi_1^*(\mc I^{\bar p}\mc C_X^*)\otimes \pi_2^* (\mc I^{\bar q}\mc C_Y^*))= H^i(R\Gamma_c(X\times Y;\pi_1^*(\mc I^{\bar p}\mc C_X^*)\otimes \pi_2^* (\mc I^{\bar q}\mc C_Y^*)))$. By Lemma \ref{L: flat}, the tensor product is the same as the left derived tensor product, and by \cite[Theorem V.10.19]{Bo}, $R\Gamma_c(X\times Y;\pi_1^*(\mc I^{\bar p}\mc C_X^*)\overset{L}{\otimes} \pi_2^* (\mc I^{\bar q}\mc C_Y^*))\cong R\Gamma_c(X;\mc I^{\bar p}\mc C_X^*)\overset{L}{\otimes}R\Gamma_c(Y;\mc I^{\bar q}\mc C_Y^*)$. 
Thus, 
\begin{align*}
\H_c^{n+m-i}&(X\times Y;\pi_1^*(\mc I^{\bar p}\mc C_X^*)\otimes \pi_2^* (\mc I^{\bar q}\mc C_Y^*) ) \\
&\cong \H_c^{n+m-i}(X\times Y;\pi_1^*(\mc I^{\bar p}\mc C_X^*)\overset{L}{\otimes} \pi_2^* (\mc I^{\bar q}\mc C_Y^*))\\
&\cong H^{n+m-i}(R\Gamma_c(X\times Y;\pi_1^*(\mc I^{\bar p}\mc C_X^*)\overset{L}{\otimes} \pi_2^* (\mc I^{\bar q}\mc C_Y^*)))\\
&\cong H^{n+m-i}(R\Gamma_c(X;\mc I^{\bar p}\mc C_X^*)\overset{L}{\otimes} R\Gamma_c(Y;\mc I^{\bar q}\mc C_Y^*))\\
&\cong \bigoplus_{r+s=n+m-i}\H_c^{r}(X;\mc I^{\bar p}\mc C_X^*)\otimes \H_c^{s}(Y;\mc I^{\bar q}\mc C_Y^*)\oplus\bigoplus_{r+s=n+m-i-1}\H_c^{r}(X;\mc I^{\bar p}\mc C_X^*)* \H_c^{s}(Y;\mc I^{\bar q}\mc C_Y^*)\\
&\cong  \bigoplus_{r+s=n+m-i}I^{\bar p}H^c_{m-r}(X)\otimes I^{\bar q}H^c_{n-s}(Y) \oplus \bigoplus_{r+s=n+m-i-1}I^{\bar p}H^c_{m-r}(X)* I^{\bar q}H^c_{n-s}(Y)\\
&\cong \bigoplus_{a+b=i}I^{\bar p}H^c_{a}(X)\otimes I^{\bar q}H^c_{b}(Y) \oplus \bigoplus_{a+b=i-1}I^{\bar p}H^c_{a}(X)* I^{\bar q}H^c_{b}(Y) \\
&\cong H_i(I^{\bar p}C^c_*(X)\otimes I^{\bar q}C^c_*(Y)).
\end{align*}
\end{proof}

Thus to show that the homology of $I^QC_*(X\times Y)$ is the homology of $I^{\bar p}C^c_*(X)\otimes I^{\bar q}C^*_c(Y)$, it suffices to show that $\mc I^Q\mc C^*$ is quasi-isomorphic to $\pi_1^*(\mc I^{\bar p}\mc C_X^*)\otimes \pi_2^* (\mc I^{\bar q}\mc C_Y^*)$.

As an intermediary for this comparison, we will use a partially-defined presheaf $R^*$ over $X\times Y$ such that if $U\subset X$ and $V\subset Y$ are open, we define $R^{n+m-*}(U\times V)=I^{\bar p}C_{*}^{\infty}(X,X-\bar U)\otimes I^{\bar q}C_{*}^{\infty}(Y,Y-\bar V)$. Even though this presheaf is defined only on sets of the form $U\times V$, restriction morphisms are nevertheless well-defined on such sets and every point has a cofinal system of neighborhoods of this form, and so an obvious modification of the usual sheafification process generates a sheaf $\mc R^*$, which is quasi-isomorphic to  $\pi_1^*(\mc I^{\bar p}\mc C_X^*)\otimes \pi_2^* (\mc I^{\bar q}\mc C_Y^*)$. As for the ordinary singular intersection chain sheaf, $\mc R^*$ is also generated up to quasi-isomorphism by the partially-defined presheaf $U\times V\to I^{\bar p}C_{*}^{c}(X,X-\bar U)\otimes I^{\bar q}C_{*}^{c}(Y,Y-\bar V)$; see \cite[Lemma 3.1]{GBF10}.

We will use the singular chain cross product  to induce homomorphisms $\phi: R^*(U\times V)\to I^QC^*(U\times V)$. We will show that $\phi$ induces a well-defined sheaf quasi-isomorphism under the hypotheses of the theorem.

We first need to determine what conditions are necessary on $Q$ for the exterior chain product $\epsilon: C_*(X) \times C_*(Y)\to C_*(X\times Y)$ (see, e.g., \cite{MK,McC}) to restrict to a well-defined homomorphism $I^{\bar p}C_*(X)\otimes I^{\bar q}C_*(Y)\to I^QC_*(X\times Y)$. For this, we claim it is sufficient to have $Q(k,l)\geq \bar p(k)+\bar q(l)$.

To verify this claim, we  note that if we have an $i$-simplex $\sigma$ in the image of a chain\footnote{Of course not every element of $I^{\bar p}C_*(X)\otimes I^{\bar q}C_*(Y)$ can be written in this form, but this group is generated by elements of this form, so it suffices to check these.} $\epsilon(\xi_1\otimes \xi_2)$, then we can consider the domain of $\sigma$ to be a simplicial simplex $\delta$  in the standard triangulation of $\Delta^a\times \Delta^b$ for some $a,b$ with $a+b=i$, and $\sigma$ is determined by restricting to $\delta$ the product $\sigma_1\times \sigma_2$ of two singular simplex maps. Here $\sigma_1,\sigma_2$ are simplices of $\xi_1,\xi_2$, respectively. Now,  $\sigma^{-1}(X_k\times Y_l)$ is the intersection of $\delta$ with  $\sigma_1^{-1}(X_k)\times \sigma_2^{-1}(Y_l)$. Since, by the allowability of $\xi_1,\xi_2$, $\sigma_1^{-1}(X_k)$ and $\sigma_2^{-1}(Y_k)$ are contained in, respectively, the $a-k+\bar p(k)$ and $b-l+\bar q(l)$ skeleta of $\Delta^a$ and $\Delta^b$ and since the product of the $r$-skeleton of $\Delta^a$ with the $s$-skeleton of $\Delta^b$ is contained in the $r+s$ skeleton of the standard  triangulation of $\Delta^a\times \Delta^b$, it follows that $\delta\cap \sigma^{-1}(X_k\times Y_l)$ is contained in the $i-k-l+\bar p(k)+\bar q(l)$ skeleton of $\delta$ and hence the $i-k-l+Q(k,l)$ skeleton of $\delta$ if $Q(k,l)\geq \bar p(k)+\bar q(l)$. The same argument applies to any simplices of the boundary, using $\bd \epsilon(\xi_1\times \xi_2)=\pm \epsilon(\bd \xi_1\otimes \xi_2\pm \xi_1\otimes\bd \xi_2)$ and the allowability of $\bd \xi_1$ and $\bd \xi_2$.

\subsection{Induction steps}

We now induct on the depth of the product $X\times Y$, i.e. the difference in dimension between $\dim(X\times Y)$ and the lowest-dimensional non-empty stratum of $X\times Y$. At each stage, we determine precisely the conditions on $Q$, $X$, and $Y$ that are necessary for $\phi: R^*(U\times V)\to I^QC^*(U\times V)$ to induce a quasi-isomorphism.

\paragraph{Depth 0.}
In the simplest case, $X$ and $Y$ are unfiltered manifolds. In this case, as is usual for intersection homology, $I^{\bar p}H_*(X;R)\cong H_*(X;R)$, since all chains are allowable, and similarly for $Y$. Since $I^QC_*$ is  an intersection homology theory with a loose perversity and there are only strata of codimension $0$ here to consider, $I^QC_*(X\times Y)$ must be $C_*(X\times Y)$, so long as $Q(0,0)\geq 0$ (otherwise there would be no allowable chains at all!). So, by the standard singular homology K\"unneth theorem, Theorem \ref{T: Kunneth} holds up through depth $0$.

\paragraph{Induction.}

Suppose now that we have proven Theorem \ref{T: Kunneth} on products of  $X\times Y$  of depth $<J$ and that $X$ and $Y$ are pseudomanifolds such that $X\times Y$ has depth $J$ and that the hypotheses of Theorem \ref{T: Kunneth} hold.

We perform a second induction over codimension of the strata. For points in the top stratum of $X\times Y$, which have euclidean neighborhoods, the isomorphism of the local stalk homology groups of $\mc I^Q\mc C^*$ and $\mc R^*$ is again just the local K\"unneth formula for manifolds. So, we assume that $\phi$ has been shown to be a quasi-isomorphism on strata of $X\times Y$ of codimension $<K$ and turn to  examining the conditions that allow us to continue on to codimension $K$. 

Suppose that $S$ is a connected stratum component of dimension $m+n-K$ in $X^m\times Y^n$. It turns out that there are two different cases to consider, depending on whether $S$ is a component of $X_{m-K}\times Y_n$ (or, symmetrically equivalently, $X_m\times Y_{n-K}$) or a component of $X_{m-k}\times Y_{n-l}$ with $k+l=K$, $k,l>0$.

\subparagraph{Case 1.}

Suppose $x\in X_{m-K}\times Y_n$, the stratum of codimension $(K,0)$. Let us first compute $H^{*}(\mc R^*_x)$. Letting $L^{K-1}$ be the link of $\pi_1(x)$ in $X$, so that $x$ has a neighborhood of the form $\R^{m-K}\times cL\times \R^n\cong \R^{n+m-K}\times cL$, 
we have
\begin{align*}
H^{*}(\mc R^*_x) &\cong \dlim_{x\in U\times V}
H_{n+m-*}(I^{\bar p}C^{c}_*(X, X-\bar U)\otimes I^{\bar q}C^{c}_*(Y,Y-\bar V))\\
&\cong H_{n+m-*}(I^{\bar p}C^{c}_*(X, X-\pi_1(x))\otimes I^{\bar q}C^{c}_*(Y,Y-\pi_2(x))).\\
\end{align*}
Since  $I^{\bar p}C^{c}_*(Y, Y-\pi_2(x))$ is flat, as follows from the proof of Lemma \ref{L: flat}, we can apply the K\"unneth theorem. Furthermore, since $\pi_2(x)\in Y_n$, $I^{\bar q}H^{c}_*(Y, Y-\pi_2(x))\cong I^{\bar q}H^c_*(\R^n,\R^n-0)$. So this simplifies to 
\begin{align}
H^{*}(\mc R^*_x)&\cong H_{m-*}(I^{\bar p}C^{c}_*(X, X-\pi_1(x)))\notag\\
&\cong I^{\bar p}H^c_{m-*}(\R^{m-K}\times cL, \R^{m-K}\times cL-\pi_1(x))\notag\\
&\cong I^{\bar p}H^c_{K-*}(cL,L\times \R)\\
&\cong 
\begin{cases}
0, &*>\bar p(K)\\
I^pH^c_{K-1-*}(L), & *\leq \bar p(K).
\end{cases}\label{E: case 1 P}
\end{align}

These isomorphisms follow from the excision, cone, and product formulas in \cite{Ki,GBF10}. (These remain valid also for loose perversities so long as $\bar p(k)\leq k-2$ for $k>0$.)

On the other hand, similarly,
\begin{align}
H^*(\mc I^Q\mc C^*_x) &\cong \dlim_{x\in U} I^QH^c_{n+m-*}(X\times Y;X\times Y-\bar U)\notag\\
&\cong I^QH^c_{n+m-*}(X\times Y;X\times Y-x)\notag\\
&\cong I^QH^c_{n+m-*}(\R^{m+n-K}\times cL, \R^{n+m-K}\times cL-x)\notag\\
&\cong I^QH^c_{K-*}(cL,L\times \R)\notag\\
&\cong 
\begin{cases}
0, &*>Q(K,0)\\
I^QH^c_{K-1-*}(L), & *\leq Q(K,0),
\end{cases}\label{E: case 1 K}
\end{align}
so long as $Q(K,0)\leq K-2$, which will be the case if $Q(K,0)=\bar p(K)$. 

Now, examining $I^QH^c_{*}(L)$, we observe that the filtration of $L$ is by the intersection of $L$ with strata of the form $X_j\times Y_n$ in $X\times Y$, $m-K< j\leq m$. In other words, $L$ has the structure of a product pseudomanifold of the form $L\times pt$, with its    perversity on the stratum $L_j\times *$, $0\leq j\leq J-1$, being that inherited from the intersection of $L$ with $X_{m-K+j+1}\times Y_n$. I.e., the perversity on $L_j\times pt$ is just $Q(K-1-j,0)$. Since we have assumed Theorem \ref{T: Kunneth} through depth $K-1$ and since $L$ has smaller depth, $I^QH^c_{*}(L)\cong H_*(I^{\bar p}C^c_*(L)\otimes I^{\bar q}C^c_*(pt))$. But the intersection homology of a point is the ordinary homology of a point, so this becomes $I^QH^c_{*}(L\times pt)\cong I^{\bar p}H_*^c(L).$

It follows from these computations that, in general, in order for Theorem \ref{T: Kunneth} to extend here, it is necessary and sufficient to have  $Q(K,0)=\bar p(K)$. 

Similarly, we establish that $Q(0,K)=\bar q(K)$.

\subparagraph{Case 2}

In this case, we assume that $x\in X_{m-k}\times Y_{n-l}$ with $k,l>0$, $k+l=K$. The difference now is that the neighborhood of $x$ has the form $\R^{m-k}\times cL^1\times \R^{n-l}\times cL_2$, which is homeomorphic to $\R^{m+n-k-l}\times cL$ with $L$ being the join $L=L_1*L_2$. 

$H^*(\mc R^*_x)$ remains fairly straightforward (aside from some index juggling) as 
\begin{align}
H^{*}(\mc R^*_x) &\cong \dlim_{x\in U\times V}
H_{n+m-*}(I^{\bar p}C^{c}_*(X, X-\bar U)\otimes I^{\bar q}C^{c}_*(Y,Y-\bar V))\notag\\
&\cong H_{n+m-*}(I^{\bar p}C^{c}_*(X, X-\pi_1(x))\otimes I^{\bar q}C^{c}_*(Y,Y-\pi_2(x)))\notag\\
&\cong \bigoplus_{a+b=n+m-*} I^{\bar p}H^{c}_a(X, X-\pi_1(x))\otimes I^{\bar q}H^{c}_b(Y,Y-\pi_2(x))\notag\\ &\qquad \oplus \bigoplus_{a+b=n+m-*-1} I^{\bar p}H^{c}_a(X, X-\pi_1(x))* I^{\bar q}H^{c}_b(Y,Y-\pi_2(x))\notag\\
&\cong \bigoplus_{a+b=n+m-*} I^{\bar p}H^{c}_{a-m+k}(cL_1,L_1)\otimes I^{\bar q}H^{c}_{b-n+l}(cL_2,L_2)\notag\\&\qquad  \oplus \bigoplus_{a+b=n+m-*-1} I^{\bar p}H^{c}_{a-m+k}(cL_1,L_1)* I^{\bar q}H^{c}_{b-n+l}(cL_2,L_2)\notag\\
&\cong \bigoplus_{r+s=*} I^{\bar p}H^{c}_{k-r}(cL_1,L_1)\otimes I^{\bar q}H^{c}_{l-s}(cL_2,L_2)\notag\\&\qquad  \oplus \bigoplus_{*=r+s-1} I^{\bar p}H^{c}_{k-r}(cL_1,L_1)* I^{\bar q}H^{c}_{l-s}(cL_2,L_2)\notag\notag\\
&\cong \bigoplus_{\overset{r+s=*}{r\leq \bar p(k), s\leq \bar q(l)}} I^{\bar p}H^{c}_{k-1-r}(L_1)\otimes I^{\bar q}H^{c}_{l-1-s}(L_2)\notag\\&\qquad  \oplus \bigoplus_{\overset{*=r+s-1}{r\leq \bar p(k), s\leq \bar q(l)}} I^{\bar p}H^{c}_{k-1-r}(L_1)* I^{\bar q}H^{c}_{l-1-s}(L_2).\label{E: product}
\end{align}
Notice  that this formula implies that $H^*(\mc R^*_x)=0$ for $*>\bar p(k)+\bar q(l)$.

Now we need to look at $H^*(\mc I^Q\mc C^*_x)$, which after the preliminary steps equivalent to those in the previous case comes down to
\begin{align}\label{E: Q trunc}
H^*(\mc I^Q\mc C^*_x)&\cong I^QH_{k+l-*}(cL,L\times \R)\notag\\
&\cong 
\begin{cases}
0,  & *> Q(k,l)\\
I^QH_{k+l-1-*}(L), &*\leq Q(k,l).
\end{cases} 
\end{align}
The cone formula applies by the Proposition \ref{P: cone}, noting that the vertex of the cone comes from the intersection of $cL$ with the stratum $X_{m-k}\times Y_{n-l}$, which has codimension pair $(k,l)$ and that, under the hypotheses of the theorem, $Q(k,l)\leq k+l-2$. 
Furthermore, even though $L$ is not a product, it inherits, as a subset of a space on which $I^QC_*$ is defined, the loose perversity intersection homology theory with $Q$ as perversity.

To compute $I^QH_{k+l-1-*}(L)$, we use that $L\cong L_1*L_2\cong (L_1\times cL_2) \cup_{L_1\times L_2} (cL_1\times L_2)$. We can then apply the Mayer-Vietoris sequence for loose perversity intersection homology theories. Each of the spaces $L_1\times L_2$, $L_1\times cL_2$, and $cL_1\times L_2$ is a product of pseudomanifolds with total depth less than $J$, and, viewed as subspaces of $X\times Y$, each inherits the  perversity conditions from $X\times Y$. Furthermore, employing Lemma \ref{L: link link}, the hypotheses of Theorem \ref{T: Kunneth} continues to hold on each of these products. 
Thus we may employ the induction hypothesis, by which Theorem \ref{T: Kunneth} holds for products of depth $<J$, and  the Mayer-Vietoris sequence looks like

\begin{diagram}
&\rTo & I^QH^c_*(L_1\times L_2)& \rTo & I^QH^c_*(cL_1\times L_2)\oplus  I^QH^c_*(L_2\times cL_2) &\rTo & I^QH^c_*(L)&\rTo,
\end{diagram} 
which, applying the product structures, becomes 
{\small\begin{diagram}
&\rTo & \underset{\oplus\bigoplus_{a+b=i-1} I^{\bar p}H_a^c(L_1)* I^{\bar q}H^c_b(L_2)}{\bigoplus_{a+b=i} I^{\bar p}H_a^c(L_1)\otimes I^{\bar q}H^c_b(L_2)} & \rTo^{i_*} & \underset{\oplus\bigoplus_{a+b=i-1} I^{\bar p}H_a^c(cL_1)* I^{\bar q}H^c_b(L_2)}{\bigoplus_{a+b=i} I^{\bar p}H_a^c(cL_1)\otimes I^{\bar q}H^c_b(L_2)}\oplus  \underset{\oplus\bigoplus_{a+b=i-1} I^{\bar p}H_a^c(L_1)* I^{\bar q}H^c_b(cL_2)}{\bigoplus_{a+b=i} I^{\bar p}H_a^c(L_1)\otimes I^{\bar q}H^c_b(cL_2)} &\rTo & I^QH^c_*(L)&\rTo,
\end{diagram} }
or, using the cone formula,

\begin{multline}\label{E: MV}
\longrightarrow\underset{\oplus\bigoplus_{a+b=i-1} I^{\bar p}H_a^c(L_1)* I^{\bar q}H^c_b(L_2)}{\bigoplus_{a+b=i} I^{\bar p}H_a^c(L_1)\otimes I^{\bar q}H^c_b(L_2)} \\
 \overset{i_*}{\longrightarrow}  \underset{\oplus\bigoplus_{\overset{a+b=i-1}{a<k-1-\bar p(k)}} I^{\bar p}H_a^c(L_1)* I^{\bar q}H^c_b(L_2)}{\bigoplus_{\overset{a+b=i}{a<k-1-\bar p(k)}} I^{\bar p}H_a^c(L_1)\otimes I^{\bar q}H^c_b(L_2)}  \oplus   \underset{\oplus\bigoplus_{\overset{a+b=i-1}{b<l-1-\bar q(l)}} I^{\bar p}H_a^c(L_1)* I^{\bar q}H^c_b(L_2)}{\bigoplus_{\overset{a+b=i}{b<l-1-\bar q(l)}} I^{\bar p}H_a^c(L_1)\otimes I^{\bar q}H^c_b(L_2)}\\
\longrightarrow I^QH^c_*(L)\longrightarrow.
\end{multline}
The maps $i_*$ here are relatively straightforward, being governed by (products of) inclusions of links into cones on links with the result that $i_*$ is an isomorphism onto a direct summand everywhere that it can be. So, for example, when $a<k-1-\bar p(k)$ and $b<l-1-\bar q(l)$ simultaneously, $i_*$ has the form of a diagonal inclusion of groups $G\into G\oplus G$, and $I^QH^c_*(L)$ inherits a term isomorphic to $G$ from the quotient. On the other hand, when 
$a\geq k-1-\bar p(k)$ and $b\geq l-1-\bar q(l)$, then $i_*$ is trivial and the relevant terms show up in $I^QH^c_{*+1}(L)$ as direct summands. In the other situations, $i_*$ is an isomorphism of summands and no contribution is made to $I^QH^c_*(L)$. The upshot is that 
\begin{align}\label{E: KQH}
I^QH_i(L) \cong &  \bigoplus_{\overset{a+b=i}{\overset{a<k-1-\bar p(k)}{b<l-1-\bar q(l)}}} I^{\bar p}H_a(L_1)\otimes  I^{\bar q}H_b(L_2) \oplus   \bigoplus_{\overset{a+b=i-1}{\overset{a<k-1-\bar p(k)}{b<l-1-\bar q(l)}}} I^{\bar p}H_a(L_1)*  I^{\bar q}H_b(L_2)\\
&
\oplus \bigoplus_{\overset{a+b=i-1}{\overset{a\geq k-1-\bar p(k)}{b\geq l-1-\bar q(l)}}} I^{\bar p}H_a(L_1)\otimes  I^{\bar q}H_b(L_2) \oplus \bigoplus_{\overset{a+b=i-2}{\overset{a\geq k-1-\bar p(k)}{b\geq l-1-\bar q(l)}}} I^{\bar p}H_a(L_1)* I^{\bar q}H_b(L_2).\notag
\end{align}

Now let's break this down into what happens in different ranges. By checking in what index ranges different possibilities  may occur, we obtain the following (compare \cite[Proposition 3]{CGJ}): 

\begin{align}
I^Q&H_i(L)\cong  \label{E: breakdown}\\
&\begin{cases}
 \displaystyle \bigoplus_{\overset{a+b=i-1}{\overset{a\geq k-1-\bar p(k)}{b\geq l-1-\bar q(l)}}} I^{\bar p}H_a(L_1)\otimes  I^{\bar q}H_b(L_2) \oplus \displaystyle \bigoplus_{\overset{a+b=i-2}{\overset{a\geq k-1-\bar p(k)}{b\geq l-1-\bar q(l)}}} I^{\bar p}H_a(L_1)* I^{\bar q}H_b(L_2),\notag\\
\hfill i\geq k+l-\bar p(k)-\bar q(l),\notag\\
I^{\bar p}H_{k-1-\bar p(k)}(L_1)\otimes  I^{\bar q}H_{l-1-\bar q(l)}(L_2), \hfill i= k+l-\bar p(k)-\bar q(l)-1,\notag\\
0,\hfill i=k+l-\bar p(k)-\bar q(l)-2,\notag\\
I^{\bar p}H_{k-2-\bar p(k)}(L_1)*I^{\bar q}H_{l-2-\bar q(l)}(L_2), \hfill i=k+l-\bar p(k)-\bar q(l)-3,\notag\\
\displaystyle \bigoplus_{\overset{a+b=i}{\overset{a<k-1-\bar p(k)}{b<l-1-\bar q(l)}}} I^{\bar p}H_a(L_1)\otimes  I^{\bar q}H_b(L_2) \oplus   \displaystyle \bigoplus_{\overset{a+b=i-1}{\overset{a<k-1-\bar p(k)}{b<l-1-\bar q(l)}}} I^{\bar p}H_a(L_1)*  I^{\bar q}H_b(L_2),\notag \\ \hfill i\leq k+l-\bar p(k)-\bar q(l)-4.\notag\\
\end{cases}
\end{align}

Next, we prepare for the comparison with $H^*(\mc R^*_x)$. Recall that we're really looking for $I^QH_{k+l-1-*}(L)$ in the range $*\leq Q(k,l)$. We plug $i=k+l-1-*$ into the above formula, simplify indices, and take $a=k-1-r$, $b=l-1-s$ to better match our computations of $H^*(\mc R^*_x)$. This yields

\begin{align}
I^Q&H_{k+l-1-*}(L)\cong\label{E: breakdown 2}\\
&\begin{cases}
 \displaystyle \bigoplus_{\overset{r+s=*}{r\leq \bar p(k),s\leq \bar q(l)}} I^{\bar p}H_{k-1-r}(L_1)\otimes  I^{\bar q}H_{l-1-s}(L_2) \oplus \displaystyle \bigoplus_{\overset{*=r+s-1}{r\leq \bar p(k),s\leq \bar q(l)}} I^{\bar p}H_{k-1-r}(L_1)* I^{\bar q}H_{l-1-s}(L_2),\notag\\
\hfill *\leq \bar p(k)+\bar q(l)-1,\notag\\
I^{\bar p}H_{k-1-\bar p(k)}(L_1)\otimes  I^{\bar q}H_{l-1-\bar q(l)}(L_2), \hfill *= \bar p(k)+\bar q(l),\notag\\
0,\hfill *=\bar p(k)+\bar q(l)+1,\notag\\
I^{\bar p}H_{k-2-\bar p(k)}(L_1)*I^{\bar q}H_{l-2-\bar q(l)}(L_2), \hfill *=\bar p(k)+\bar q(l)+2,\notag\\
\displaystyle \bigoplus_{\overset{r+s=*}{r\leq \bar p(k),s\leq \bar q(l)}} I^{\bar p}H_{k-1-r}(L_1)\otimes  I^{\bar q}H_{l-1-s}(L_2) \oplus   \displaystyle \bigoplus_{\overset{r+s=*-1}{r\leq \bar p(k),s\leq \bar q(l)}} I^{\bar p}H_{k-1-r}(L_1)*  I^{\bar q}H_{l-1-s}(L_2),\notag\\  \hfill *\geq \bar p(k)+\bar q(l)+3.\notag\\
\end{cases}
\end{align}

Now, we know by \eqref{E: Q trunc} that $H^*(\mc I^Q\mc C^*_x)$ will be given by this formula for $*\leq Q(k,l)$ and by $0$ for $*>Q(k,l)$, and we need this to agree with formula \eqref{E: product} for $H^*(\mc R^*_x)$. What must our $Q(k,l)$ be for this to happen?

Assuming $Q(k,l)$ is sufficiently large, it is clear that we have complete agreement between our formulas for $H^*(\mc I^Q\mc C^*_x)$ and $H^*(\mc R^*_x)$ for $*\leq\bar p(k)+\bar q(l)-1$. It is less clear, though still true, that we have agreement for $*=\bar p(k)+\bar q(l)$. In this case, we have $H^*(\mc I^Q\mc C^*_x)\cong I^QH_{k+l-1-\bar p(k)-\bar q(l)}(L)\cong I^{\bar p}H_{k-1-\bar p(k)}(L_1)\otimes  I^{\bar q}H_{l-1-\bar q(l)}(L_2)$. On the other hand, $$H^{\bar p(k)+\bar q(l)}(\mc R^*_x)\cong \bigoplus_{\overset{r+s=\bar p(k)+\bar q(l)}{r\leq \bar p(k), s\leq \bar q(l)}} I^{\bar p}H^{c}_{k-1-r}(L_1)\otimes I^{\bar q}H^{c}_{l-1-s}(L_2)   \oplus \bigoplus_{\overset{\bar p(k)+\bar q(l)=r+s-1}{r\leq \bar p(k), s\leq \bar q(l)}} I^{\bar p}H^{c}_{k-1-r}(L_1)* I^{\bar q}H^{c}_{l-1-s}(L_2),$$
but, looking at the index restrictions, the only term that doesn't vanish is again $I^{\bar p}H_{k-1-\bar p(k)}(L_1)\otimes  I^{\bar q}H_{l-1-\bar q(l)}(L_2)$.

Now, as previously observed, $H^*(\mc R^*_x)=0$ for $*>\bar p(k)+\bar q(l)$. So to obtain agreement between $H^*(\mc I^Q\mc C^*_x)$ and $H^*(\mc R^*_x)$, we now only need to have $H^*(\mc I^Q\mc C^*_x)=0$ in this range. Remarkably, there is more than one option for $Q(k,l)$ that makes this work out, thanks to the $0$ sitting in the middle of the join formula. Essentially, we can assign $Q(k,l)$ to truncate on either side of this $0$ without causing disagreement between $H^{*}(\mc I^Q\mc C^*_x)$ and $H^{*}(\mc R^*_x)$, meaning that we can choose $Q(k,l)=\bar p(k)+\bar q(l)$ or $Q(k,l)=\bar p(k)+\bar q(l)+1$. 

We can stretch things even further if we are working with field coefficients or if we know, more generally, that 
$I^{\bar p}H_{k-2-\bar p(k)}(L_1;R)*I^{\bar q}H_{l-2-\bar q(l)}(L_2;R)=0$, in other words, if condition (2c) of Theorem \ref{T: Kunneth} is in play. In this case, $H^{*}(\mc I^Q\mc C^*_x)=0=H^{*}(\mc R^*_x)$ for all $*$ even when $Q(k,l)=\bar p(k)+\bar q(l)+2$.  

We see that, in general,  Theorem \ref{T: Kunneth} will not hold if $Q(k,l)>\bar p(k)+\bar q(l)+2$. 

Technically speaking, up to this point we have only shown that $\mc I^Q\mc C^*$ and $\mc R^*$ have abstractly isomorphic stalk homology groups. In order to have a proper quasi-isomorphism, we should also address the map $\phi$. But having computed the various homology groups that arise, we see that there are no surprises. Chasing back through the isomorphisms, the relevant elements of $H^*(\mc R^*_x)$ are represented, roughly speaking, by the tensor and torsion products of chains of the form $c\xi_1$ and $c\xi_2$, where $\xi_i$ is a cycle in $IC_*(L_i)$. More formally, looking for the moment just at the tensor product terms in the appropriate dimension ranges, we have cycles $\xi_i\in IC_*(L_i)$ whose cones represent cycles in $IC_*(cL_i,cL_i\times \R)$. The tensor product elements of $H^*(\mc R^*)$ are represented by the chain products of these $c\xi_i$ in $cL_1\times cL_2$. The reader can then check, by working back through the Mayer-Vietoris sequence computations, that this product chain, as a chain of $\mc I^Q\mc C^*_x$, also represents the corresponding homology class in  $H^*(\mc I^Q\mc C^*_x)$. The idea for the torsion product terms is precisely the same, though with more technical details surrounding how the torsion product terms of the algebraic K\"unneth theorem are represented by appropriate tensor products of chains. The reader equipped with \cite[Section 58]{MK} should have no trouble working out the analogous details. 

Finally, we thus  conclude that our desired quasi-isomorphism exists given the conditions of Theorem \ref{T: Kunneth}. We also see from the local computations  that we cannot, in general, expand the range of $Q$ further without imposing stronger conditions on the vanishing of terms of the form $I^{\bar p}H_a(L_1)\otimes I^{\bar q}H_b(L_2)$. 

This completes the induction and thus the proof of Theorem \ref{T: Kunneth}.\hfill\qedsymbol

\section{Super loose perversities}\label{S: super}

Theorem \ref{T: Kunneth} requires that perversities satisfy $\bar p(k)\leq k-2$ for all $k>0$. In this section, we look at what happens when we allow $\bar p(k)>k-2$ for some $k$. Generalizing the definition of  \cite{GBF10}, we refer to any such perversities as ``super.''

First, we should note that there are two ways to deal with superperversities. The first way, following King \cite{Ki}, is to define intersection homology exactly as usual. However, the cone formula, quoted in Proposition \ref{P: cone}, needs to be modified. The more general version is:

\begin{proposition}[King]\label{P: cone 2}
Let $L$ be an $n-1$ dimensional filtered space with coefficient system $R$, and let $\bar p$ be a loose perversity. Then 
\begin{equation*}
I^{\bar p}H^{c}_i(cL;R) \cong
\begin{cases}
0, & i\geq n-1-\bar p(n), i\neq 0,\\
I^{\bar p}H^{c}_{i}(L;R), & i<n-1-\bar p(n),\\
R, & i=0, \bar p(n)>n-2.
\end{cases}
\end{equation*}
If $L$ is compact, then 
\begin{equation*}
I^{\bar p}H^c_i(cL,L\times \R;R)\cong I^{\bar p}H^{\infty}_i(cL;R) \cong
\begin{cases}
I^{\bar p}\td H_{i-1}(L;R), & i\geq n-\bar p(n),\\
0, & i<n-\bar p(n),
\end{cases}
\end{equation*}
\end{proposition}
Notice the key changes from Proposition \ref{P: cone}. In the first formula, if $0\geq n-1-\bar p (n)$, we have an $R$ instead of the expected $0$ when $i=0$, and in the second formula, we must use reduced intersection homology, denoted $I^{\bar p}\td H_{*}$ (just as for ordinary homology, the reduction simply eliminates an $R$ summand in dimension $0$).

The other alternative for superperversities is to modify the definition of intersection homology in such a way that the original cone formula of Proposition \ref{P: cone} is preserved. There seem to be two ways to do this: with the stratified coefficient systems $R_0$ we introduced in \cite{GBF10} and with Saralegi's chain complex $S^{\bar p}C_*(X,X_{\bar t-\bar p})$, introduced in \cite{Sa05}. In fact, these turn out to be equivalent theories, as we demonstrate below in the Appendix.
The former theory was developed to provide a singular chain model for the superperverse intersection homology appearing in the Cappell-Shaneson superduality theorem \cite{CS}, while the latter was developed to prove a version of Poincar\'e duality for the de Rham version of intersection cohomology with nontraditional perversities.
It is shown in \cite{GBF10} that $I^{\bar p}H_*(X;R_0)$ is equivalent to the Deligne sheaf intersection homology if $\bar p$ is a superperversity, and, furthermore, it follows from \cite{GBF10} that if $\bar p(n)\leq  n-2$ for all $n\geq 2$, in particular if $\bar p$ is a traditional perversity, then $I^{\bar p}H_*(X;R)=I^{\bar p}H_*(X;R_0)$. Thus, $I^{\bar p}H_*(X;R_0)$ is also an extension of the traditional intersection homology.  All of the  standard properties of intersection homology (including excision, Mayer-Vietoris sequences, the K\"unneth theorem with $\R^n$, and the homotopically fine sheaf) continue to hold with $R_0$ coefficients. It is also possible to slightly modify the $R_0$ theory up to quasi-isomorphism so that the associated sheaves are flat (see the Appendix).

We discuss these alternative versions of intersection homology in slightly more detail in the Appendix. The main point, however, is that due to the properties listed in the last paragraph, our K\"unneth Theorem, Theorem \ref{T: Kunneth}, continues to hold in this setting.

\begin{theorem}
Theorem \ref{T: Kunneth} holds for  superperversities if we replace  $R$ with the stratified coefficient system $R_0$ of \cite{GBF10} or if we replace $I^{\bar p}C_*(X;R)$  with Saralegi's $S^{\bar p}C_*(X,X_{\bar t-\bar p};R)$ and $I^QC_*(X\times Y;R)$ with an appropriate $I^QC_*(X\times Y, (X\times Y)_{\bar t-\bar p};R)$.
\end{theorem}
\begin{proof}
Using $R_0$ coefficients, the proof of Theorem \ref{T: Kunneth} goes through \emph{mutatis mutandis}, and by Proposition \ref{P: saralegi}, below, the two alternate versions of intersection homology are equivalent. 
\end{proof}

This leaves consideration of what happens when we employ superperversities without making any of these modiciations to the definition of intersection homology. We first observe that there is no need to consider \emph{all} possible superperversities:

\begin{lemma}\label{L: too super}
Let $\bar p$ be a loose perversity and $X$ a pseudomanifold. Define $\check p$ by $\check p(0)=0$ and $\check p(k)=\min\{\bar p(k), k-1\}$ for $k>0$. Then $I^{\bar p}H_*(X)\cong 
I^{\check p}H_*(X)$. 
\end{lemma}
\begin{proof}
Since $\check p(k)\leq \bar p(k)$ for all $k$, there are natural inclusions $I^{\check p}C_*\subset I^{\bar p}C_*$, which induce maps of sheaves $\mc I^{\check p}\mc C^*\to \mc I^{\bar p}\mc C^*$. Using the K\"unneth formula for a product with $\R^n$ and the cone formula Proposition \ref{P: cone 2} to perform local computations, it is easy to check that this inclusion is a quasi-isomorphism. 
\end{proof}

Thus, in some sense, allowing $\bar p(k)=k-1$ is the \emph{only} superperversity, and  we may assume that $\bar p(k)\leq k-1$ for $k>0$ without loss of generality in computing $I^{\bar p}H_*$.

Next, we show that Theorem \ref{T: Kunneth} will not hold in general in this setting.
 
Consider the following example. Suppose we have a product space $Z=cX\times cY$, where $X$ and $Y$ are respectively $k-1$ and $l-1$ dimensional pseudomanifolds. The $0$-dimensional stratum of $Z$, the product of the cone vertices, has codimension $(k,l)$. Let $\bar p(k)>k-1$, and let the coefficient ring be $\Z$. Then, using the loose perversity cone formula of Proposition \ref{P: cone 2}, $I^{\bar p}H_*^c(cX)$ is trivial except for a $\Z$ in degree $0$. Thus $H_*(I^{\bar p}C^c_*(cX)\otimes I^{\bar q}C_*^c(cY))\cong I^{\bar q}H_*^c(cY)$, which, in general, will vanish only for $*\geq l-1-\bar q(l)$.

On the other hand, to compute $I^QH_*^c(Z)$, we use that $Z\cong cL$, where $L$ is the join $X*Y$ and has dimension $k+l-1$. So, $I^QH_i^c(Z)=0$ for $i\geq k+l-1-Q(k,l)$ (so long as $Q(k,l)\leq k+l-2$). Now, as in Theorem \ref{T: Kunneth}, let $Q(k,l)=\bar p(k)+\bar q(l)+I$, where $I$ stands for $0$, $1$, or $2$. Then, with $\bar p(k)=k-1$, we have $k+l-1-Q(k,l)=k+l-1-\bar p(k)-\bar q(l)-I=l-\bar q(l)-I$. We will have a contradiction with the end result of the last paragraph if this number is less than $l-1-\bar q(l)$, which will occur if $I=2$. Also, note that we will indeed have $Q(k,l)=\bar p(k)+\bar q(l)+2=k+\bar q(l)+1 \leq k+l-2$ in this case, so long as $\bar q(l)\leq l-3$, which is certainly possible.
Thus the $I=2$ case of Theorem \ref{T: Kunneth} cannot occur in general if $\bar p$ is a superperversity.

In general, the best we can do to generalize Theorem \ref{T: Kunneth} for superperversities  with $R$ coefficients, is the following.

\begin{theorem}
\begin{enumerate}
\item If $\bar p(k)\leq k-1$ for all $k>0$, $\bar q(l)\leq l-2$ for all $l>0$, then  Theorem \ref{T: Kunneth} holds with standard coefficients $R$ except that  $Q(k,l)=\bar p(k)+\bar q(l)+2$ is never allowed. Similarly if  $\bar p(k)\leq k-2$ for all $k>0$, $\bar q(l)\leq l-1$ for all $l>0$.

\item If $\bar p(k)\leq k-1$ for all $k>0$, $\bar q(l)\leq l-1$ for all $l>0$, then  Theorem \ref{T: Kunneth} holds with standard coefficients $R$ except that  neither $Q(k,l)=\bar p(k)+\bar q(l)+2$ nor $Q(k,l)=\bar p(k)+\bar q(l)+1$ is  allowed.
\end{enumerate}

In general, we can not extend Theorem \ref{T: Kunneth} further.
\end{theorem}

The reader should note in the following proof that it does not hold if we allow  $\bar p(k)>k-1$. However,  bear in mind  Lemma \ref{L: too super}.

\begin{proof}
We note the necessary modifications to the proof of Theorem \ref{T: Kunneth}.

The first modification to the preceding proof is simple: In Case 1, we should use the appropriate reduced intersection homology theories in the relevant places in equations \eqref{E: case 1 P} and \eqref{E: case 1 K}.

The changes necessary in Case 2 are more substantial. We begin by assuming we are in the first case of the theorem, i.e. $\bar p(k)\leq k-1, \bar q(l)\leq l-2$ for $k,l>0$.

 Firstly, we need to used reduced intersection homology in equation \eqref{E: product}. Note, however, that using reduced intersection homology in the torsion product terms of \eqref{E: product} is irrelevant, since both $I^{\bar p}H_0$ and $I^{\bar p}\td H_0$ will be free $R$ modules so that torsion products involving them will vanish.
 
We do not need to modify  \eqref{E: Q trunc} since with the assumptions of this case of the theorem, $Q(k,l)\leq k+l-2$.

Next, we must look at the Mayer-Vietoris sequence \eqref{E: MV}. 
Here,  we now need to observe that  $I^{\bar p}H_0(cL_1;R)=R$, instead of the expected $0$ if the previous computation were to hold. This results in additional terms in the middle term of the sequence, in particular an $R\otimes I^{\bar q}H_{i}(L_2)$. There is no new   $R* I^{\bar q}H_{i-1}(L_2)$, since this is $0$ as $R$ is a free $R$-module. How these influence $I^QH_*(L)$ depends on $\bar q$ and $i$:

\begin{enumerate}
\item If $i<l-1-\bar q(l)$, then an $R\otimes I^{\bar q}H_i(L_2)$ term already exists in the second summand of the middle term of the sequence, so the extra one gets pushed into $I^QH_i(L)$, as in our prior computations in the proof of Theorem \ref{T: Kunneth}. 

\item If, on the other hand, $i\geq l-1-\bar q(l)$, the only $R\otimes I^{\bar q}H_i(L_2)$ summand in the middle term is that coming from the $cL_1$ term. So $i_*$ maps onto this one term, which leaves one less  $R\otimes I^{\bar q}H_i(L_2)$ for $I^QH_{i+1}(L)$. So, the $I^{\bar p}H_0(L_1)\otimes I^{\bar q}H_{i}(L_2)$ that would ordinarily appear in $I^QH_{i+1}(L)$ becomes $I^{\bar p}\td H_0(L_1)\otimes I^{\bar q}H_{i}(L_2)$.
\end{enumerate}

So, in summary, if $i<l-1-\bar q(l)$, $I^QH_i(L)$  picks up a term  $R\otimes I^{\bar q}H_i(L_2)$, and  if $i\geq l-\bar q(l)$, an $I^{\bar p}H_0(L_1)\otimes I^{\bar q}H_{i}(L_2)$ becomes an $I^{\bar p}\td H_0(L_1)\otimes I^{\bar q}H_{i}(L_2)$. If $i=l-1-\bar q(l)$, then $I^QH_i(L)$  remains unchanged. 

Now, looking at \eqref{E: breakdown} and \eqref{E: breakdown 2} and comparing with the proof of Theorem \ref{T: Kunneth}, the only thing we have to check now is that the top three lines of these formulas convert in the current case to the reduced intersection homology so that we will have agreement with \eqref{E: product}. But these are the lines of \eqref{E: breakdown} corresponding to $i\geq k+l-\bar p(k)-\bar q(l)-2$. Using $\bar p(k)=k-1$, this is $i\geq l-\bar q(l)-1$. But by what we have just worked out, this is precisely the range where we get the desired reduced intersection homology terms, for $i\geq l-\bar q(l)-1$, while for $i= l-\bar q(l)-1$, the $0$ remains. We also see that when $i=k+l-\bar p(k)-\bar q(l)-3$, We pick up extra, possibly non-torsion, terms, so we cannot extend to $Q(k,l)=\bar p(k)+\bar q(l)+2$ (as already noted above). 

The arguments when $\bar q(l)=l-1$ but $\bar p(k)\leq k-2$ are the same.

Finally, we consider $\bar p(k)=k-1$ and $\bar q(l)=l-1$. Now all of the intersection homology groups in \eqref{E: product} become reduced.  

To compute $I^QH_*(L)$, we note in the Mayer-Vietoris sequence that, using Proposition \ref{P: cone 2}, the middle terms in degree $i$ are $R\otimes I^{\bar q}H_i(L_2)$ and $I^{\bar p}H_i(L_1)\otimes R$. If $i>0$, these terms are distinct, and so for $i\geq 2$, $$I^QH_i(L)\cong \displaystyle\bigoplus_{a+b=i-1} I^{\bar p}\td H_a^c(L_1)\otimes I^{\bar q}\td H^c_b(L_2)\oplus\displaystyle\bigoplus_{a+b=i-2} I^{\bar p}\td H_a^c(L_1)* I^{\bar q}\td H^c_b(L_2).$$ On the other hand, in degree $0$,
if we think of the explicit $R$s as being generated by basepoints in $L_1$ and $L_2$, then the only repeat amongst the terms $R\otimes I^{\bar q}H_0(L_2)$ and $I^{\bar p}H_0(L_1)\otimes R$ is that corresponding to $R\otimes R$ (basepoint times basepoint), which occurs once in each of these summands. The rest all correspond to unique terms of $I^{\bar p}H_0(L_1)\otimes H^{\bar q}H_0(L_2)$ and so are in the image of $i_*$ in the Mayer-Vietoris sequence. So, $I^QH_0(L)=R$ corresponding to the cokernel of the  diagonal image of $R\otimes R$ in $(R\otimes R)\oplus (R\otimes R)$. Lastly, by counting, if $I^{\bar q}H_0(L_2)\cong R^s$ and $I^{\bar p}H_0(L_1)\cong R^t$, then $I^QH_1(L)\cong R^{ts-t-s+1}\cong R^{(t-1)(s-1)}$, so we also have $I^QH_1(L)\cong I^{\bar p}\td H_0(L_1)\otimes I^{\bar q}\td H_0(L_2)$. 

Thus 
\begin{align*}
H^*(\mc I^Q\mc C^*_x)&\cong I^QH_{k+l-*}(cL,L\times \R)\\
&\cong
\begin{cases}
0,  & *> Q(k,l)=k+l-2,\\
I^QH_{k+l-1-*}(L), &*\leq Q(k,l)=k+l-2,
\end{cases} \\
&\cong
\begin{cases}
0,  & *> k+l-2,\\
\underset{\oplus\bigoplus_{a+b=k+l-3-*} I^{\bar p}\td H_a^c(L_1)* I^{\bar q}\td H^c_b(L_2)}{\bigoplus_{a+b=k+l-2-*} I^{\bar p}\td H_a^c(L_1)\otimes I^{\bar q}\td H^c_b(L_2)}, &*\leq k-l-2,
\end{cases} \\
&\cong 
\begin{cases}
0,  & *> k+l-2,\\
\underset{\oplus\bigoplus_{*=r+s-1} I^{\bar p}\td H_{k-1-r}^c(L_1)* I^{\bar q}\td H^c_{l-1-s}(L_2)}{\bigoplus_{r+s=*} I^{\bar p}\td H_{k-1-r}^c(L_1)\otimes I^{\bar q}\td H^c_{l-1-s}(L_2)}, &*\leq k-l-2.
\end{cases}
\end{align*}
Comparing with \eqref{E: product}, the theorem holds in this case.

Finally, the theorem cannot hold if we allow $\bar p(k)=k-1$, $\bar q(l)=l-1$, and $Q(k,l)=\bar p(k)+\bar q(l)+1=k+l-1$. In this case, $H^{k+l-1}(\mc I^Q\mc C^*)\cong I^QH_0(L)\cong R$, by the preceding computations. But, looking at \eqref{E: product}, $H^{k+l-1}(\mc P^*_x)$ vanishes  because $r+s=k+l-1$ forces either $r>k-1$ or $s>l-1$. 

\end{proof}

\appendix

\section{Appendix: $I^{\bar p}C_*(X;R_0)$ and $S^{\bar p}C_*(X;X_{\bar t-\bar p})$}

In this appendix, we discuss in slightly more detail the modified intersection homology theories $I^{\bar p}H_*(X;R_0)$ and $I^{\bar p}H_*(X;X_{\bar t-\bar p})$. Both of these theories arose with fundamentally the same purpose: to create a version of intersection homology for the which the cone formula of Proposition \ref{P: cone} remains valid even if $\bar p(k)>k-2$ for some $k$. 

Saralegi's chain complex $S^{\bar p}C_*(X;X_{\bar t-\bar p})$ is defined as 
$$S^{\bar p}C_*^c(X,X_{\bar t-\bar p})=\frac{\left(A^{\bar p}C_*(X)+AC_*^{\bar p+1}(X_{\bar t-\bar p})\right)\cap \bd^{-1}\left(A^{\bar p}C_{*-1}(X)+AC_{*-1}^{\bar p+1}(X_{\bar t-\bar p})\right)}{AC_*^{\bar p+1}(X_{\bar t-\bar p})\cap \bd^{-1}AC_{*-1}^{\bar p+1}(X_{\bar t-\bar p})},$$
where $\bar t$ is the top perversity, $\bar t(k)=k-2$, $A^{\bar p}C_i(X)$ is generated by the $\bar p$-allowable $i$-simplices of $X$, $X_{\bar t-\bar p}$ is the closure of the union of the singular strata $S$ of $X$ such that $\bar t(S)-\bar p(S)<0$, and $A^{\bar p}C_i(X_{\bar t-\bar p})$ is generated by the $\bar t-\bar p$ allowable $i$-simplices  with support in $X_{\bar t-\bar p}$. Here, the singular strata $S$ of $X$ are those contained within $X^{n-1}$. Also, for determining allowability, $(\bar t-\bar p)(0)$ is defined to be $\geq 0$ and not $\bar t(0)-\bar p(0)=-2-\bar p(0)$ (equivalently, Saralegi defines allowability of chains by only checking the allowability conditions with respect to the singular strata).

To describe $I^{\bar p}C_*(X;R_0)$,  $R_0$ is defined in \cite{GBF10} to consist of the pair of coefficient systems given by $R$  on $X-X^{n-1}$ and the constant $0$ system on $X^{n-1}$. Then, given a singular simplex $\sigma: \Delta \to X$, in \cite{GBF10} we defined  a coefficient of $\sigma$ in $R_0$ to consist of a lift of $\sigma|_{\sigma^{-1}(X-X^{n-1})}$ to the coefficient system $R$ over $X-X^{n-1}$ together with a trivial $0$ coefficient on $ \sigma^{-1}(X^{n-1})$. Using these kinds of coefficients, $I^{\bar p}H_*(X;R_0)$ is defined in the usual way. The point is that certain simplices ``die off'' if they live completely in  $X^{n-1}$, and this is sufficient for the cone formula of Proposition \ref{P: cone} to hold even for superperversities. See \cite{GBF10} for more details. 

We here make  one minor modification from this definition in \cite{GBF10} in that we will assume that all coefficient lifts of  $\sigma|_{\sigma^{-1}(X-X^{n-1})}$ will be globally constant. In other words, to each simplex whose image does not lie completely in $X^{n-1}$, a coefficient assigns a single element of $R$ to all points of $\sigma^{-1}(X-X^{n-1})$.  This assumption allows us to avoid oddities such as singular simplices carrying infinite amounts of coefficient data on infinite numbers of connected components of $\sigma^{-1}(X-X^{n-1})$. With this assumption, each $C^c_i(X;R_0)$ will be a free $R$-module generated by those simplices whose support is not contained in $X^{n-1}$, and $I^{\bar p}C^c_i(X;R_0)$ will be a submodule. This modification ensures that Lemma \ref{L: flat} continues to hold in this setting. Even with this assumption, Proposition \ref{P: cone} continues to hold. 

In fact, this also implies that, for constant coefficients, this modification does not change anything up to quasi-isomorphism over the original definition of $I^{\bar p}C_*(X;R_0)$. To see this, note that  if we let $I^{\bar p}C_*(X;\td R_0)$ denote the version of the intersection chain complex with $R_0$ coefficients as defined in \cite{GBF10}, there is an obvious inclusion $I^{\bar p}C_*(X;R_0)\into I^{\bar p}C_*(X;\td R_0)$. This inclusion is certainly  an isomorphism if $X$ is an unstratified manifold, and it is now  easy to induct on depth and use the K\"unneth theorem for which one term is $\R^n$ along with the cone formula Proposition \ref{P: cone}, which holds for both versions $R_0$ and $\td R_0$, to conclude that the inclusion is a quasi-isomorphism for any $X$. 

Finally, we show that our modified version of $I^{\bar p}C_*(X;R_0)$ and $S^{\bar p}C_*(X,X_{\bar t-\bar p})$ are isomorphic.

\begin{proposition}\label{P: saralegi}
$I^{\bar p}C_*^c(X;R_0)$ is isomorphic to Saralegi's $S^{\bar p}C_*^c(X,X_{\bar t-\bar p})$.
\end{proposition}
\begin{proof}
We fix the ground ring as $R$ throughout, occasionally leaving it tacit.

We will first construct a homomorphism $f: I^{\bar p}C_*^c(X;R_0) \to S^{\bar p}C_*^c(X,X_{\bar t-\bar p})$. Then we will show that $f$ is an isomorphism for each fixed degree. Lastly, we show that $f$ is in fact a chain map.

To define $f$, let $\xi\in I^{\bar p}C_i^c(X;R_0)$. Then, by definition, $\xi\in A^{\bar p}C_i(X;R)$. Furthermore, recall that $\bd \xi\in  I^{\bar p}C_{i-1}^c(X;R_0)$ can be described by taking the boundary of $\xi$ in $C_*(X;R)$ and then setting the coefficients to all simplices with support in $X^{n-1}$ to $0$. The remaining simplices of $\bd \xi$  must be $\bar p$-allowable. But this means that in $ C_{i-1}(X;R)$, $\bd \xi=\zeta+\eta$, where $\zeta\in A^{\bar p}_{i-1}(X;R)$ and $\eta\in C_*(X^{n-1};R)$. We claim that, in fact, $\eta\in  AC_{*-1}^{\bar p+1}(X_{\bar t-\bar p})$. Let $S\subset X_{n-k}$, $k>0$, be a stratum such that $|\eta|\cap S\neq \emptyset$. Then $\dim(|\eta|\cap S)\leq \dim(|\xi|\cap S)\leq i-k+\bar p(S)=i-1-k+\bar p(S)+1$. Thus, $\eta$ is $\bar p+1$ allowable in  $X^{n-1}$. Furthermore, suppose that $S\not \subset X_{\bar t-\bar p}$ (continuing to assume $S\subset X_{n-k}$). Then $\bar t(S)-\bar p(S)=k-2-\bar p(S)\geq 0$, so $\bar p(S)\leq k-2$. Thus $\dim(|\xi|\cap S)\leq i-k+k-2=i-2$. Thus the interior of no simplex of $\eta$ can lie in $S$. So, the interiors of the simplices of $\eta$ must be in $X_{\bar t-\bar p}$, and thus $|\eta|\subset X_{\bar t-\bar p}$, since $X_{\bar t-\bar p}$ is closed. So $\eta\in A^{\bar p+1}C_*(X_{\bar t-\bar p})$. It follows that $\xi$ represents an element of $S^{\bar p}C_*^c(X,X_{\bar t-\bar p})$. 

This assignment taking $\xi$ to an element of  $S^{\bar p}C_*^c(X,X_{\bar t-\bar p})$ is clearly additive, so this defines our homomorphism $f$. 

Suppose $\xi \in I^{\bar p}C_*(X;R_0)$ and $f(\xi)=0$. Then $f(\xi)\in AC_*^{\bar p+1}(X_{\bar t-\bar p})\cap \bd^{-1}AC_{*-1}^{\bar p+1}(X_{\bar t-\bar p})\subset C_*(X^{n-1})$, and so $\xi$ is $0$ in $I^{\bar p}C_*(X;R_0)$. Thus $f$ is injective. 

On the other hand, suppose $z\in S^{\bar p}C_*^c(X,X_{\bar t-\bar p})$. We can represent $z$ by $z=x+y$, where $x\in A^{\bar p}C_*(X)$ and $y\in AC_*^{\bar p+1}(X_{\bar t-\bar p})$. We claim that we also have $\bd y\in AC_{*-1}^{\bar p+1}(X_{\bar t-\bar p})$. Certainly $\bd y\in C_{*-1}(X_{\bar t-\bar p})$, so the issue is just the allowability. We know that $\bd z=\bd x+\bd y\in A^{\bar p}C_{*-1}(X)+AC_{*-1}^{\bar p+1}(X_{\bar t-\bar p})$, so, in particular, it is $\bar p+1$ allowable. Since $x$ is $\bar p$ allowable, for any stratum $S\subset X_{n-k}$, we have $\dim(|\bd x|\cap S)\leq \dim(|x|\cap S)\leq i-k+\bar p(S)=i-1-k+\bar p(S)+1$, so that $\bd x$ is $\bar p+1$ allowable. Thus it follows that $\bd y=\bd z-\bd x$ is $\bar p+1$ allowable. So $\bd y \in AC_{*-1}^{\bar p+1}(X_{\bar t-\bar p})$. So, $y\in AC_*^{\bar p+1}(X_{\bar t-\bar p})\cap \bd^{-1}AC_{*-1}^{\bar p+1}(X_{\bar t-\bar p})$. It follows that $z$ and $x$ represent the same element of $S^{\bar p}C_*^c(X,X_{\bar t-\bar p})$. Now, $x\in A^{\bar p}C_*(X)$, and we know that $\bd x=\bd z-\bd y\in A^{\bar p}C_{*-1}(X)+AC_{*-1}^{\bar p+1}(X_{\bar t-\bar p})$, which implies that if we set to zero the coefficients of the simplices of $\bd x$ with support in $X^{n-1}$ (and, in particular, in $X_{\bar t-\bar p}$), then what remains will be $\bar p$ allowable. So $x$ fits the description of a chain in  $I^{\bar p}C_*(X;R_0)$.   Thus $f$ is surjective.

Finally, to see that $f$ is a chain map, consider $\xi\in I^{\bar p}C_*(X,R_0)$. As noted above, we can write $\bd \xi$ in $C_*(X;R_0)$ as $\bd \xi=\zeta+\eta$, and $\zeta\in A^{\bar p}C_{i-1}(X;R)$ represents $\bd \xi$ in $I^{\bar p}C_*(X;R_0)$. So $f(\bd \xi)$ is represented by $\zeta$. On the other hand, we have $\bd f(\xi)$ represented by $\zeta+\eta$. But we have also seen that $\eta\in AC_{*-1}^{\bar p+1}(X_{\bar t-\bar p})$. So the decomposition $\bd \xi=\zeta+\eta$ has the same form as the decomposition $z=x+y$ consided in the last paragraph, and thus by those arguments we know that $\zeta+\eta$ and $\zeta$ represent the same element of $S^{\bar p}C_*(X,X_{\bar t-\bar p})$.  Thus $f$ is a chain map.

\end{proof}

\bibliographystyle{amsplain}
\bibliography{bib}

\providecommand{\bysame}{\leavevmode\hbox to3em{\hrulefill}\thinspace}
\providecommand{\MR}{\relax\ifhmode\unskip\space\fi MR }
\providecommand{\MRhref}[2]{%
  \href{http://www.ams.org/mathscinet-getitem?mr=#1}{#2}
}
\providecommand{\href}[2]{#2}
\begin{thebibliography}{10}

\bibitem{BaIH}
Markus Banagl, \emph{Topological invariants of stratified spaces}, Springer
  Monographs in Mathematics, Springer-Verlag, New York, 2006.

\bibitem{Bo94}
Armand Borel, \emph{Introduction to middle intersection cohomology and perverse
  sheaves}, Algebraic groups and their generalizations: classical methods
  (University Park, PA, 1991), Proc. Sympos. Pure Math., vol.~56, Amer. Math.
  Soc., Providence, RI, 1994, pp.~25--52.

\bibitem{Bo}
A.~Borel~et. al., \emph{Intersection cohomology}, Progress in Mathematics,
  vol.~50, Birkhauser, Boston, 1984.

\bibitem{CS}
Sylvain~E. Cappell and Julius~L. Shaneson, \emph{Singular spaces,
  characteristic classes, and intersection homology}, Annals of Mathematics
  \textbf{134} (1991), 325--374.

\bibitem{Chee80}
J.~Cheeger, \emph{On the {H}odge theory of {R}iemannian pseudomanifolds},
  Geometry of the Laplace Operator (Proc. Sympos. Pure Math., Univ. Hawaii,
  Honolulu, Hawaii, 1979) (Providence, RI), vol.~36, Amer. Math. Soc., 1980,
  pp.~91--146.

\bibitem{CGJ}
Daniel~C. Cohen, Mark Goresky, and Lizhen Ji, \emph{On the {K}{\"u}nneth
  formula for intersection cohomology}, Trans. Amer. Math. Soc. \textbf{333}
  (1992), 63--69.

\bibitem{GBF17}
Greg Friedman, \emph{Intersection homology and {P}oincar\'e duality on
  homotopically stratified spaces}, submitted;.

\bibitem{GBF13}
\bysame, \emph{Intersection homology of stratified fibrations and
  neighborhoods}, to appear in Advances in Math.; see also
  http://arxiv.org/abs/math.GT/0701112.

\bibitem{GBF18}
\bysame, \emph{On the chain-level intersection pairing for {PL}
  pseudomanifolds}, submitted; see also http://arxiv.org/abs/0808.1749.

\bibitem{GBF3}
\bysame, \emph{Stratified fibrations and the intersection homology of the
  regular neighborhoods of bottom strata}, Topology Appl. \textbf{134} (2003),
  69--109.

\bibitem{GBF11}
\bysame, \emph{Superperverse intersection cohomology: stratification
  (in)dependence}, Math. Z. \textbf{252} (2006), 49--70.

\bibitem{GBF10}
\bysame, \emph{Singular chain intersection homology for traditional and
  super-perversities}, Trans. Amer. Math. Soc. \textbf{359} (2007), 1977--2019.

\bibitem{GM1}
Mark Goresky and Robert MacPherson, \emph{Intersection homology theory},
  Topology \textbf{19} (1980), 135--162.

\bibitem{GM2}
\bysame, \emph{Intersection homology {II}}, Invent. Math. \textbf{72} (1983),
  77--129.

\bibitem{GS83}
Mark Goresky and Paul Siegel, \emph{Linking pairings on singular spaces},
  Comment. Math. Helvetici \textbf{58} (1983), 96--110.

\bibitem{HS91}
Nathan Habegger and Leslie Saper, \emph{Intersection cohomology of cs-spaces
  and {Z}eeman's filtration}, Invent. Math. \textbf{105} (1991), 247--272.

\bibitem{Ki}
Henry~C. King, \emph{Topological invariance of intersection homology without
  sheaves}, Topology Appl. \textbf{20} (1985), 149--160.

\bibitem{KirWoo}
Frances Kirwan and Jonathan Woolf, \emph{An introduction to intersection
  homology theory. {S}econd edition}, Chapman \& Hall/CRC, Boca Raton, FL,
  2006.

\bibitem{Klei}
Steven Kleiman, \emph{The development of intersection homology theory}, A
  Century of Mathematics in America Part {II} (Providence, R.I.), Hist. Math.,
  vol.~2, Amer. Math. Soc., 1989, pp.~543--585.

\bibitem{LANG}
Serge Lang, \emph{Algebra: Revised third edition}, Springer-Verlag, New York,
  2002.

\bibitem{McC}
J.E. McClure, \emph{On the chain-level intersection pairing for {PL}
  manifolds}, Geom. Topol. \textbf{10} (2006), 1391--1424.

\bibitem{MK}
James~R. Munkres, \emph{Elements of algebraic topology}, Addison-Wesley,
  Reading, MA, 1984.

\bibitem{Q2}
Frank Quinn, \emph{Intrinsic skeleta and intersection homology of weakly
  stratified sets}, Geometry and topology (Athens, GA, 1985), Lecture Notes in
  Pure and Appl. Math., vol. 105, Dekker, New York, 1987, pp.~225--241.

\bibitem{Sa05}
Martintxo Saralegi-Aranguren, \emph{de {R}ham intersection cohomology for
  general perversities}, Illinois J. Math. \textbf{49} (2005), no.~3, 737--758
  (electronic).

\bibitem{Si83}
P.H. Siegel, \emph{Witt spaces: a geometric cycle theory for {KO}-homology at
  odd primes}, American J. Math. \textbf{110} (1934), 571--92.

\end{thebibliography}

Several diagrams in this paper were typeset using the \TeX\, commutative
diagrams package by Paul Taylor.

\end{document}